\documentclass[12pt]{amsart}

\usepackage{amsfonts}
\usepackage{hyperref}
\usepackage{amsbsy}
\usepackage{amsmath}
\usepackage{amssymb}
\usepackage{amscd, cite} 
\usepackage{enumitem}
\usepackage[mathscr]{eucal}
\numberwithin{equation}{section}

\newtheorem{theorem}{Theorem}[section]
\newtheorem{main}{Theorem}

\newtheorem{lemma}[theorem]{Lemma}
\newtheorem{proposition}[theorem]{Proposition}

\theoremstyle{definition}
\newtheorem{hypothesis}[theorem]{Hypothesis}
\newtheorem{definition}[theorem]{Definition}

\numberwithin{equation}{section}
\makeatletter
\let\c@theorem\c@equation
\makeatother

\renewcommand{\leq}{\leqslant}
\renewcommand{\geq}{\geqslant}
\newcommand{\norm}{\trianglelefteq}
\newcommand{\cin}{\subseteq}

\newcommand{\nin}{\notin}

\newcommand{\set}[1]{\{#1\}}

\newcommand{\gen}[1]{\langle #1 \rangle}
\newcommand{\ol}{\overline}
\renewcommand{\bar}{\ol}

\newcommand{\inv}[1]{#1^{-1}}

\newcommand{\resbar}[2]{#1|_{#2}}

\newcommand{\F}{\mathcal{F}}
\newcommand{\C}{\mathcal{C}}
\newcommand{\E}{\mathcal{E}}

\newcommand{\D}{\mathcal{D}}
\newcommand{\K}{\mathcal{K}}
\newcommand{\M}{\mathcal{M}}
\newcommand{\N}{\mathcal{N}}

\newcommand{\I}{\mathcal{I}}
\renewcommand{\L}{\mathcal{L}}
\renewcommand{\H}{\mathcal{H}}

\renewcommand{\phi}{\varphi}

\newcommand{\elem}{\mathscr{E}}

\newcommand{\id}{\operatorname{id}}

\newcommand{\Hom}{\operatorname{Hom}}

\newcommand{\Iso}{\operatorname{Iso}}

\newcommand{\Aut}{\operatorname{Aut}}
\newcommand{\Out}{\operatorname{Out}}
\newcommand{\Inn}{\operatorname{Inn}}

\newcommand{\Syl}{\operatorname{Syl}}
\newcommand{\Baum}{\operatorname{Baum}}

\newcommand{\hyp}{\mathfrak{hyp}}
\newcommand{\foc}{\mathfrak{foc}}

\title{A characterization of the $2$-fusion system of $L_4(\lowercase{q})$}
\author[J.~Lynd]{Justin Lynd}
\address{Department of Mathematics\\ 
                Rutgers University\\
                110 Frelinghuysen Rd\\
                Piscataway, NJ  08854}
\email{jlynd@math.rutgers.edu}
\date{\today}

\begin{document}
\begin{abstract}
We study a saturated fusion system $\F$ on a finite $2$-group $S$ having a
Baumann component based on a dihedral $2$-group.  Assuming $\F = O^2(\F)$,
$O_2(\F) = 1$, and the centralizer of the component is a cyclic $2$-group, it
is shown that $\F$ is uniquely determined as the $2$-fusion system of
$L_4(q_1)$ for some $q_1 \equiv 3 \pmod{4}$. This should be viewed as a
contribution to a program recently outlined by M. Aschbacher for the
classification of simple fusion systems at the prime $2$.  The corresponding
problem in the component-type portion of the classification of finite simple
groups (the $L_2(q)$, $A_7$ standard form problem) was one of the last to be
completed, and was ultimately only resolved in an inductive context with heavy
artillery. Thanks primarily to requiring the component to be Baumann, our main
arguments by contrast require only $2$-fusion analysis and transfer. We deduce
a companion result in the category of groups.
\end{abstract}
\maketitle

Saturated fusion systems (or Frobenius categories) are categories, defined by
Puig, codifying simultaneously the properties of $G$-conjugacy of $p$-subgroups
in a finite group, and of Brauer pairs associated to a $p$-block of a group
algebra for $G$. The study of $p$-fusion in finite groups began in the last
decade of the 19th century with Burnside and Frobenius. In the latter half of
the twentieth, the analysis of fusion was indispensable for the classification
of finite simple groups. Abstract fusion theory has evolved in the last decade
via the work of many to become the foundation for investigation of spaces which
behave like $p$-completed classifying spaces of finite groups as well as a
natural setting for studying the $p$-local structure of finite groups.  In the
latter setting, the structure theory of saturated fusion systems parallels that
of finite groups. A saturated fusion system has appropriate analogues of
$O_p(G)$, $O^p(G)$ and $O^{p'}(G)$ \cite{BCGLO2007}, a transfer map
\cite{BrotoLeviOliver2003}, normal subgroups and quotients, simplicity,
components, the layer $E(G)$, and the generalized Fitting subgroup
\cite{AschbacherNormal, AschbacherGeneralized}.

With this groundwork in place, Aschbacher has proposed
\cite[Section~II.13-15]{AschbacherKessarOliver2011} a program for the
classification of simple fusion systems at the prime $2$, and has begun to
carry out substantial parts of it \cite{AschbacherGeneration, AschbacherS3Free,
AschbacherF2type, AschbacherFSofCT}. See also work of Henke \cite{Henke2011}
and Welz \cite{Welz2012}. One reason for doing this is to effect a directed
search for new exotic $2$-fusion systems other than the Solomon systems. A more
central aim is to simplify portions of the classification of finite simple
groups. Working in the category of fusion systems provides a clean separation
of the analysis of fusion in a finite group from other considerations, such as
dealing with the obstructions caused by cores of local subgroups (which are not
present in the fusion system setting), and of group recognition from local
structure. This suggests such simplifications might be realized predominantly
in the component-type case. Early examples of this can be seen in Aschbacher's
$E$-balance theorem \cite[Theorem~7]{AschbacherGeneralized} and the Dichotomy
Theorem \cite[II.14.3]{AschbacherKessarOliver2011} for saturated fusion
systems. 

Denote by $J(S)$ the Thompson subgroup of $S$, the subgroup generated by
elementary abelian subgroups of $S$ of maximum rank. The Baumann subgroup
$\Baum(S)$ is $C_S(\Omega_1(ZJ(S)))$ \cite[B.2.2]{AschbacherSmith2004I}. Let
$\nu_2$ be the $2$-adic valuation. In this paper, we make a contribution to
Aschbacher's program by carrying out a standard form problem for the $2$-fusion
system of $L_2(q)$ ($q$ odd). The result is the following characterization of
the $2$-fusion system of $L_4(q_1)$ ($q_1 \equiv 3 \pmod{4}$). 

\begin{main}\label{T:main}
Let $\F$ be a saturated fusion system on the $2$-group $S$ and let $x \in S$
be a fully $\F$-centralized involution.  Set $\C = C_\F(x)$ and $T = C_S(x)$.
Let $\K$ be a subsystem of $\C$.  Assume $\F = O^2(\F)$, $O_2(\F) = 1$, and the
following three items:
\begin{enumerate}[label=\textup{(\arabic{*})}]
\item $\K$ is a perfect normal subsystem of $\C$ on a dihedral group of order $2^k$, 
\item $Q := C_T(\K)$ is cyclic, and
\item $\Baum(S) \leq T$.
\end{enumerate}
Then $S \cong D_{2^k} \wr C_2$, and $\F \cong \F_S(G)$ where $G \cong L_4(q_1)$
for some $q_1 \equiv 3 \pmod{4}$ with $\nu_2(q_1+1) = k-1$. 
\end{main}
\begin{proof}
From Lemma~\ref{L:2rank3or4}, $S$ is of $2$-rank $3$ or $4$.
Proposition~\ref{P:2centralfinal} says that $x$ does not lie in the center of
$S$, while Theorems~\ref{T:2rank3} and \ref{T:Q>2} show that $S$ is of $2$-rank
$4$ and $Q$ is of order at least $4$. Finally, Theorem~\ref{T:2rank4main}
identifies $\F$ and completes the proof of the theorem.
\end{proof}

A $2$-fusion system $\K$ is said to be \emph{perfect} if $\K = O^2(\K)$. Thus,
in Theorem~\ref{T:main}, $\K$ is determined as the unique simple saturated
fusion system on a nonabelian dihedral group of the given order, i.e. as
$\F_2(L_2(q))$ for some (any) $q \equiv \pm 1 \pmod{8}$ with $\nu_2(q^2 - 1) =
k+1$. Hypotheses (1) and (2) are equivalent to specifying the structure of the
generalized Fitting subsystem of $\C$ as $F^*(\C) = Q \times \K$, and they
imply that $\K$ is a \textit{standard subsystem} in the sense of
\cite{AschbacherFSofCT}. See also \cite{AschbacherCopenhagen} for an expository
outline of some of the results in \cite{AschbacherFSofCT}. 
Matthew Welz considers a standard form problem for the fusion system of
$L_2(q)$ in a situation complementary to Theorem~A, where $C_T(\K)$ has
$2$-rank at least $2$, and assuming (1) (but not (3)); see \cite{Welz2012}.

Let $S$ be an arbitrary finite $2$-group, $\F$ a saturated fusion system over
$S$, and $W$ a weakly $\F$-closed subgroup of $S$. Then $\F$ is of
$W$-\emph{characteristic} $2$-\emph{type} if $N_\F(P)$ is constrained for every
fully $\F$-normalized $P \leq S$ with $W \leq N_S(P)$, and of
$W$-\emph{component type} if there is a fully $\F$-centralized involution $x$
with the property that $W \leq C_S(x)$ and $C_\F(x)$ has a component. In the
latter case, we say that a subsystem $\K$ is a $W$-component if it is a
component in some $C_\F(x)$ with the property that $W \leq C_S(x)$. 

Hypothesis (3) is the statement that $\K$ is a $\Baum(S)$-component and thus
$\F$ is of $\Baum(S)$-component type (alternatively, Baumann component type).
The reason one might want to make this restriction in the context of a
classification of simple $2$-fusion systems is detailed in \cite[\S
II.14]{AschbacherKessarOliver2011}. However, see \cite{AschbacherFSofCT} for
updates to the outline of the program, which have the effect of requiring
Theorem~A to be eventually strengthened. As regards our specific situation, (3)
allows us to avoid building the Sylow $2$-subgroups of certain larger groups in
characteristic $2$.  Without (3) for instance, the shadows of fusion systems of
$\Aut(L_5(2))$, $\Aut(U_5(2))$, and $\Aut(Sp_4(4))$ cause difficulties, as
these have involutory automorphisms fixing $L_2(9) \cong \Omega_5(2) \cong
Sp_4(2)'$. In the group case, these, $HS$, and $\Aut(He)$ are identified by
Fritz \cite{Fritz1977} and (independently) by Harris-Solomon
\cite{HarrisSolomon1977} and Harris \cite{Harris1977}. All of these almost
simple groups are of component type. However, all but $\Aut(He)$ (which has a
Baumann component isomorphic to a four-fold cover of $L_3(4)$) are of Baumann
characteristic $2$-type, as are their $2$-fusion systems. More seriously, (3)
permits us to avoid an inductive approach like that taken by Harris
\cite{Harris1981}, where a $K$-group hypothesis was required together with the
solution to a large number of other standard form problems.  It was here that
the natural targets, $L_4(q^{\frac{1}{2}})$ ($q^\frac{1}{2} \equiv 3 \pmod{4}$)
and $U_4(q^{\frac{1}{2}})$ ($q^\frac{1}{2} \equiv 1 \pmod{4})$) appeared
(together with, e.g., $\Aut(\Omega_8^{-}(q^{\frac{1}{4}})$) via an appeal to
Aschbacher's Classical Involution Theorem. We note that $L_4(q_1)$ and
$U_4(q_2)$ have equivalent fusion systems at the prime $2$ whenever their Sylow
$2$-subgroups are isomorphic; this follows from a more general theorem due to
Broto, M\o ller, and Oliver \cite[Theorem~3.3]{BrotoMollerOliver2012}.

For the heart of the arguments in this paper, only elementary 2-group analysis,
fusion, and transfer are required.  Use of transfer is made via the
Thompson-Lyons transfer lemma for fusion systems \cite{LyndTL}; see Subsection
\ref{SS:transfer}.  At the beginning of the analysis, we encounter a difficulty
unique to the fusion system setting in getting ahold on the structure of
subsystems of $\C/C_\C(\K)$ containing (an isomorphic copy) of $\K$, which a
priori contain among them exotic extensions of $\K$.  However, work of
Andersen, Oliver, and Ventura \cite{AOV2012} shows that this is in fact not the
case provided certain higher limits of functors associated to $\K$ vanish; see
Subsection~\ref{S:tame}. In the middle, our presentation in the $2$-central and
$2$-rank $3$ cases follows closely the treatment in \cite{GLS6}, primarily for
the reason that hypothesis (3) of Theorem~\ref{T:main} provides little or no
restriction in these cases. At the end, after determining the structure of $S$,
we apply a piece of Oliver's classification \cite{OliverSectRank4} of fusion
systems on $2$-groups of sectional rank at most $4$ in order to identify $\F$.  

When combined with a theorem of David Mason \cite{Mason1973} and Glauberman's
$Z^*$-theorem, we obtain the following companion of Theorem~\ref{T:main} in the
category of groups. Recall that $Z^*(G)$ is the preimage in $G$ of
$Z(G/O_{2'}(G))$. 
\begin{main}
\label{T:maingrp}
Let $G$ be a fusion simple finite group \textup{(}i.e. with $G = O^2(G)$ and
$Z^*(G) = O_{2'}(G)$\textup{)}, $S \in \Syl_2(G)$, and $x \in \Omega_1(ZJ(S))$.
Assume
\begin{enumerate}[label=\textup{(\arabic{*})}]
\item $C = C_G(x)$ has a perfect normal subgroup $K$ with $K/O_{2'}(K) \cong
    L_2(q)$ \textup{(}$q \equiv \pm 1 \pmod{8}$\textup{)} or $A_7$, and
\item $C_C(K/O_{2'}(K))$ has cyclic Sylow $2$-subgroups.
\end{enumerate}
Then $K/O_{2'}(K) \cong L_2(q)$, $q$ is a square, and $S$ is uniquely
determined by $\nu_2(q^2-1)$. Moreover, $O^{2'}(G/O_{2'}(G)) \cong
L_4(q^{\frac{1}{2}})$ if $q^{\frac{1}{2}} \equiv 3 \pmod{4}$, while
$O^{2'}(G/O_{2'}(G)) \cong U_4(q^{\frac{1}{2}})$ if $q^{\frac{1}{2}} \equiv 1
\pmod{4}$.  
\end{main}

The proof of Theorem~\ref{T:maingrp}, assuming Theorem~\ref{T:main}, is found
near the beginning of Section~\ref{S:centralizer}.

\subsection*{Notation}
Homomorphisms are applied on the right. We prefer to write conjugation-like
maps in the exponent. For instance, the image of an element $s \in S$ (or
subgroup $P \leq S$) under a morphism $\phi$ in a fusion system is denoted
$s^\phi$ (or $P^\phi$). 
\begin{itemize}
\item $C_n$ is the cyclic group of order $n$, $D_{2^n}$ ($n \geq 2$), $Q_{2^n}$
($n \geq 3)$, $SD_{2^n}$ ($n \geq 4$) are the dihedral, quaternion, and
semidihedral groups, respectively, of order $2^n$.
\item $G^\#$ is the set of nonidentity elements of $G$.
\item $\mathcal{I}_p(G)$ is the set of elements of $G$ of order $p$.
\item $\Omega_n(P) = \gen{x \in P \mid x^{p^n} = 1}$.
\item $\mho^{n}(P) = \gen{x^{p^n} \mid x \in P}$.
\item $\mathscr{E}_{p^n}(P)$ is the set of elementary abelian subgroups of $P$
of order $p^n$.
\item Nonstandardly, for $P$ with $Z(P)$ of order $2$, $P \wr_*
C_2$ is the
quotient of $P \wr C_2$ by its center (a wreathed commuting product).
\end{itemize}

\subsection*{Acknowledgements} 
This work was carried out in large part in the author's Ph.D. thesis at Ohio
State, and was supported by a dissertation fellowship from the Graduate School
at Ohio State. The author would like to express his sincere gratitude to Ron
Solomon for all his help and encouragement over the years. Richard Lyons read
an early version of this paper and offered many probing questions and helpful
comments, for which the author is very appreciative.  The author would like to
thank Michael Aschbacher and Bob Oliver for their willingness in making
available early versions of their papers, as well as Matthew Welz for
stimulating discussions.

\section{Preliminary results}\label{S:background}

General references for group theoretic material are \cite{Gorenstein1980},
\cite{Suzuki1982}, and \cite{GLS6}.

\subsection{Automorphism groups of $p$-groups}\label{SS:general}

We first list some results about automorphism groups of $p$-groups needed later.

\begin{theorem}\label{T:stabnormalseries}
Let $A$ be a $p'$-group of automorphisms of the $p$-group $S$ which stabilizes
a normal series $1 = S_0 \leq S_1 \leq \cdots \leq S_n = S$ and acts trivially
on each factor $S_{i+1}/S_i$. Then $A = 1$.
\end{theorem}
\begin{proof}
See for example \cite[Theorem~3.2]{Gorenstein1980}.
\end{proof}

A finite group is \emph{indecomposable} if it is not the direct product of two
proper subgroups.
\begin{proposition}\label{P:krullschmidt}
An automorphism of a direct product of indecomposable finite groups permutes
the commutator subgroups of the factors.
\end{proposition}
\begin{proof}
This is a consequence of the Krull-Schmidt theorem for finite groups, found in
\cite[Theorem~2.4.8]{Suzuki1982}.  A proof of the current statement is given in
\cite[Proposition~3.1]{OliverSplitting}.  
\end{proof}

The next lemma describes the outer automorphism group of a nonabelian dihedral
$2$-group $D$, and lists a couple of additional statements which express the
fact that a noncentral involution of $D$ cannot be a commutator or a square in
a $2$-group containing $D$ as a normal subgroup.  

\begin{lemma}\label{L:autD}
Let $D$ be a $2$-group isomorphic to $D_{2^{k+1}}$ for some $k \geq 2$. Fix the
presentation $\gen{b,c \mid b^2 = c^{2^k} = 1, b^{-1}cb = c^{-1}}$ for $D$ and
let $C = \gen{c}$ be the cyclic maximal subgroup of $D$. Let $S$ be any
$2$-group containing $D$ as a normal subgroup. Then
\begin{enumerate}[label=\textup{(\alph{*})}]
\item $\Out(D) \cong A \times B$ where $A \cong C_2$ and $B \cong C_{2^{k-2}}$
is the kernel of the action of $\Out(D)$ on the $D$-classes of four subgroups of $D$.
\item $[S,S] \leq C_S(C)$, and 
\item if $S_0$ is the preimage of $\Omega_1(S/C_S(D)D)$ in $S$, then
$\mho^1(S_0) \leq C_S(D)C$.
\end{enumerate}
\end{lemma}
\begin{proof}
For (a), $A$ may be generated by the class $[\eta]$ of the automorphism $\eta$
sending $b \mapsto c^{-1}b$ and inverting $c$, and $B$ may be generated by the
class $[\phi]$ of $\phi$ centralizing $b$ and sending $c \mapsto c^5$. 

For (b), fix a $2$-group $S$ containing $D$ as a normal subgroup. Then $C \norm
S$ as $C$ is a characteristic subgroup of $D$.  Hence (b) follows from the
exact sequence $1 \to C_S(C) \to S \to \Aut_S(C) \to 1$ and the fact that
$\Aut(C)$ is abelian.

Let $S_0$ be the preimage of $\Omega_1(S/C_S(D)D)$ in $S$ as in (c).  Thus,
$S_0$ consists of the elements of $S$ which square into $C_S(D)D$. Let $s \in
S_0$. If $s \in C_S(D)D$ we have that $s$ squares into $C_S(D)\mho^1(D) =
C_S(D)\mho^1(C)$, so we may assume that $s$ induces a nontrivial (involutory) outer
automorphism of $D$. If $s$ induces some element of the coset $B[\eta]$ as in
(a), then $s$ centralizes no noncentral involution of $D$, and hence $s^2 \in
C_S(D)C$ as claimed. So we may further assume that $s$ induces an involutory
outer automorphism in $B$. Then $c^s = cz$ or $c^{-1}z$, and so $s^2 \in
C_{C_S(D)D}(C) = C_S(D)C$ as claimed.  
\end{proof}

\begin{lemma}\label{L:autDwr2}
Suppose $k \geq 3$ and let $D$ be a $2$-group isomorphic to $Q_{2^{k+1}}$,
$SD_{2^{k+1}}$, $C_2 \times D_{2^k}$, $D_{2^k} \times D_{2^k}$, or $D_{2^k} \wr
C_2$.  Then $\Aut(D)$ is a $2$-group.  
\end{lemma}
\begin{proof}
This follows from Lemma~\ref{L:autD} and \cite[23.3]{AschbacherFGTSecond}
after applying Theorem~\ref{T:stabnormalseries} to an appropriate normal
series of $D$. 
\end{proof}

\subsection{Fusion systems}\label{SS:fusion}

We assume familiarity with basic definitions and results regarding fusion
systems, as can be found in \cite{AschbacherKessarOliver2011} and
\cite{CravenTheory}, and our notation follows the former for the most part. We
also assume the reader has a working knowledge of Aschbacher's normal
subsystems \cite{AschbacherNormal}, components, and generalized Fitting
subsystem \cite{AschbacherGeneralized}.  

For a group $G$, write $c_g\colon x \mapsto g^{-1}xg$ for the conjugation
homomorphism induced by $g \in G$. For subgroups $H$ and $K$ denote by
$\Hom_G(H,K) = \{c_g \mid g^{-1}Hg \leq K\}$ the set of group homomorphisms
from $H$ to $K$ induced by conjugation by elements of the group $G$. Write
$\Aut_G(H)$ for $\Hom_G(H,H)$. When $\phi\colon H \to K$ is any isomorphism, we
write the induced map from $\Aut(H) \to \Aut(K)$ as $\alpha \mapsto
\alpha^\phi$.

We refer to \cite[I.2.5]{AschbacherKessarOliver2011} for the definition of a
\emph{saturated} fusion system in the form used in this paper. Thus $\F$ is
saturated if it satisfies the Sylow and extension axioms. The extension axiom
says that given an isomorphism $\phi \in \Hom_\F(P,Q)$ with $Q$ fully
$\F$-centralized, there is a morphism $\tilde{\phi} \in \Hom_\F(N_\phi, S)$,
such that $\tilde{\phi}|_P = \phi$. Here,
\[
N_\phi = \{s \in N_S(P) \mid (c_s)^\phi \in \Aut_S(Q)\}. 
\]
Note $C_S(P)P \leq N_\phi$ for any such isomorphism $\phi$, so every map $P \to
Q$ with $Q$ fully $\F$-centralized must extend to $C_S(P)P$; we will apply the
extension axiom in this special case quite often.  A fusion system of a finite
group is saturated \cite[Theorem~2.3]{AschbacherKessarOliver2011}.

We will often say an element $x \in S$ is fully $\F$-centralized if $\gen{x}$
is fully $\F$-centralized, especially when $x$ is an involution.  Following
Aschbacher, we will sometimes write $\F^c$, $\F^{r}$, and $\F^{f}$ for the set
of $\F$-centric, $\F$-radical, and fully $\F$-normalized subgroups of $S$,
respectively. Concatenation in the superscript denotes the intersection of the
relevant sets. For example, $\F^{cr}$ is the set of subgroups which are both
$\F$-centric and $\F$-radical.  A saturated fusion system $\F$ is determined by
the $\F$-automorphism groups of the subgroups which lie in $\F^{fcr}$. This is
Alperin's fusion theorem for saturated fusion systems
\cite[A.10]{BrotoLeviOliver2003}.

\begin{lemma}[Burnside]\label{L:BurnsideLemma}
Let $\F$ be a saturated fusion system on the $p$-group $S$, and suppose that
$T$ is a weakly $\F$-closed subgroup of $S$. Then any morphism in $\F$ between
subgroups of $Z(T)$ lies in $N_\F(T)$. 
\end{lemma}
\begin{proof}
Suppose $P$ and $Q$ are subgroups of $Z(T)$, and let $\phi \in \Iso_\F(P,Q)$.
Let $\psi \in \Iso_\F(Q,Q')$ with $Q'$ fully $\F$-centralized.  By the
extension axiom, $\phi\psi$ and $\psi$ have extensions to $C_S(P)$ and $C_S(Q)$
respectively, and these subgroups both contain $T$. Restricting these
extensions to $T$ and using the fact that $T$ is weakly $\F$-closed, we get
automorphisms $\alpha$, $\beta \in \Aut_\F(T)$ such that $(\alpha\beta^{-1})|_P
= \phi$, which is what was to be shown.
\end{proof}

The most important weakly closed subgroups for our purposes are the Thompson
subgroup $J(S)$ generated by elementary abelian subgroups of $S$ of maximum
rank, and the Baumann subgroup $\Baum(S) = C_S(\Omega_1(Z(J(S))))$
\cite[B.2.2]{AschbacherSmith2004I}. Each of these are weakly closed in any
fusion system over $S$. 

We now turn to a discussion of the hyperfocal subsystem $O^p(\F)$ and the
residual subsystem $O^{p'}(\F)$ of a saturated fusion system.  The focal
subgroup of $\F$ is defined as
\[
\foc(\F) = \gen{[s,\phi] \mid \phi \in \Hom_{\F}(\gen{s},
S)} \leq S.
\]
The hyperfocal subgroup of $\F$ is
\[
\hyp(\F) =  \gen{[s,\phi]
\mid s \in P \leq S \mbox{ and } \varphi \in O^p(\Aut_\F(P))} \leq S.
\]
These are analogues of $[G,G] \cap S$ and $O^p(G) \cap S$.

\begin{lemma}\label{L:fochyp}
Let $\F$ be a saturated fusion system on the $p$-group $S$ and let $T$ be a
strongly $\F$-closed subgroup of $S$. Then
\begin{enumerate}[label=\textup{(\alph{*})}]
\item $\hyp(\F) \leq \foc(\F)$ and both subgroups are strongly $\F$-closed,
\item the quotient $\F/T$ is the fusion system of the $p$-group $S/T$ if and
only if $T \geq \hyp(\F)$, 
\item the quotient $\F/T$ is the fusion system of the abelian $p$-group $S/T$
if and only if $T \geq \foc(\F)$, and
\item if $S/\foc(\F)$ is cyclic, then $\foc(\F) = \hyp(\F)$.
\end{enumerate}
\end{lemma}
\begin{proof}
Part (a) is straightforward.  We present a concise proof for (b) (and (c)) due
to Craven.  By Alperin's fusion theorem, $\F/T$ is the fusion system of the
$p$-group $S/T$ if and only if there are no $p'$-automorphisms in $\F/T$ of
subgroups of $S/T$.  Under the surjective morphism $\F \to \F/T$, this happens
if and only if, for each subgroup $P$ of $S$, each $p'$-automorphism $\alpha$
of $P$, we have $[P, \alpha] \leq T$.  Since $O^p(\Aut_\F(P))$ is generated by
the $p'$ elements in $\Aut_\F(P)$, we conclude that $\F/T$ is the fusion system
of $S/T$ if and only if $T \geq \hyp(\F) = \gen{[P, O^p(\Aut_\F(P))] \mid P
\leq S}$. A similar argument establishes that in addition, $S/T$ is abelian
$p$-group if and only if $T \geq \foc(\F)$.

Now suppose $S/\foc(\F)$ is cyclic as in (d). Set $S^+ = S/\hyp(\F)$ and $\F^+
= \F/\hyp(\F)$. Then parts (b), (c) and Theorem~E of \cite{Craven2010} imply that
$\foc(\F)^+ = \foc(\F^+)$, and the latter is just the commutator subgroup
$[S^+,S^+]$ because $\F^+$ is the fusion system of the $p$-group $S^+$.
Therefore, the commutator quotient $S^+/[S^+,S^+] = S^+/\foc(\F)^+ \cong
S/\foc(\F)$ is cyclic. It then follows from \cite[5.1.2]{Gorenstein1980} that
$[S^+,S^+]=1$, i.e. $\foc(\F) = \hyp(\F)$ as claimed.  
\end{proof}

There is exactly one saturated subsystem of $\F$ for each overgroup $T$ of
$\hyp(\F)$ in $S$ \cite[Theorem~4.3]{BCGLO2007}.  Such a subsystem is normal if
and only if $T$ is normal in $S$ \cite[Section~7]{AschbacherGeneralized}. The
subsystem over $\hyp(\F)$ is the \emph{hyperfocal subsystem} and is denoted
$O^{p}(\F)$. 

We use in an essential way following theorem of Aschbacher, which allows one to
consider the (internal) product of a $p$-group with a normal subsystem. See
also \cite{Henke2013} for a simplification of Aschbacher's construction and
proof of saturation.
\begin{theorem}[{\cite[8.21]{AschbacherGeneralized}}]\label{T:Aschppower}
Let $\F$ be a saturated fusion system on the $p$-group $S$ and $\F_0$ a normal
subsystem on the strongly $\F$-closed subgroup $S_0 \leq S$.  For each subgroup
$X$ of $S$ containing $S_0$, there is a saturated subsystem ${\F_0}X$ of $\F$
with the following properties.
\begin{enumerate}[label=\textup{(\alph{*})}]
\item $\F_0 \norm \F_0X$,
\item $\F_0X/S_0 \cong \F_{X^+}(X^+)$ where $X^+ = X/S_0$, and
\item the map $Y \mapsto \F_0Y$ is a bijection between the set of subgroups $Y
\leq S$ containing $S_0$ and the set of saturated subsystems of $\F_0S$
containing $\F_0$.
\end{enumerate}
\end{theorem}
In the particular case where $\F_0 = O^p(\F_0)$, Theorem~\ref{T:Aschppower}
gives that $\F_0 = O^p(\F_0X)$, since $\hyp(\F_0) \leq \hyp(\F_0X) \leq S_0$ by
(b) and Lemma~\ref{L:fochyp}(b).

In \cite[Theorem~5.4]{BCGLO2007} (see also \cite[Theorem~6.11]{Puig2006}) the
authors give a description of the fusion subsystems of ``index prime to $p$'':
they are in one-to-one correspondence with overgroups of a certain subgroup
$\Gamma \leq \Aut_\F(S)$ containing $\Aut_S(S)$, and hence there is a unique
minimal one, the \emph{residual subsystem} of $\F$, denoted $O^{p'}(\F)$. The
following corollary of this result will suffice for our purposes.

\begin{proposition}\label{C:pprimeindexcor}
Let $S$ be a finite $p$-group with automorphism group a $p$-group. Then $\F =
O^{p'}(\F)$ for every saturated fusion system $\F$ on $S$.
\end{proposition}

The following lemma lists the relationship between the hyperfocal and residual
subsystems, surjective morphisms, and direct products
\cite[I.6.5]{AschbacherKessarOliver2011} we will need. 

\begin{lemma}\label{L:O^pbasic}
Suppose that $\F$ is a saturated fusion system on the $p$-group $S$. 
\begin{enumerate}[label=\textup{(\alph{*})}]
\item If $\theta\colon \F \to \F^+$ is a surjective morphism of fusion systems, then
$\theta(O^p(\F)) = O^p(\F^+)$ and $\theta(O^{p'}(\F)) = O^{p'}(\F^+)$.
\item If $\F = \F_1 \times \F_2$ with each $\F_i$ saturated, then $O^p(\F)
= O^{p}(\F_1) \times O^p(\F_2)$ and $O^{p'}(\F) = O^{p'}(\F_1) \times
O^{p'}(\F_2)$. 
\end{enumerate}
\end{lemma}
\begin{proof}
As a surjective morphism of fusion systems is surjective on morphisms, part (a)
follows from Alperin's fusion theorem and the fact that $O^p(G)$ and
$O^{p'}(G)$ are fully invariant subgroups.  Part (b) is
\cite[Proposition~3.4]{AOV2012}.
\end{proof}

Later in Section~\ref{S:2rank4Q>2}, we will make use of the following special
case of a more general theorem of Oliver.

\begin{theorem}[\!{\cite[Theorem~C]{OliverSplitting}}]\label{T:OliverSplitting}
Let $\F = O^2(\F)$ be a saturated fusion system on a direct product $D_1 \times
D_2$ of two nonabelian dihedral $2$-groups of the same order. Then $\F = \F_1
\times \F_2$ where $\F_i = O^2(\F_i)$ is a fusion system on $D_i$. 
\end{theorem}

The hypotheses of Oliver's theorem require that $\F = O^{2'}(\F)$, but this
holds in the case stated above by Lemma~\ref{L:autDwr2} and
Proposition~\ref{C:pprimeindexcor}.  Also, we must have $\F_i = O^2(\F_i)$ by
Lemma~\ref{L:O^pbasic}(b).

\subsection{Centralizers}\label{SS:centralizer}

A suitable definition of the centralizer of an arbitrary subsystem of a fusion
system is not yet available.  However, Aschbacher has shown that in the case
that a subsystem $\E$ is normal in $\F$, one can define the centralizer
$C_\F(\E)$, which enjoys many of the properties one would like.  This system is
based on a strongly $\F$-closed subgroup $C_S(\E)$.

\begin{theorem}[\!\!{\cite[(6.7)]{AschbacherGeneralized}}]\label{T:C_S(E)}
Let $\F$ be a saturated fusion system on $S$ and let $\E$ be a normal subsystem
on $T$. Let $\mathcal{X}$ denote the set of subgroups $X \leq C_S(T)$ for which
$C_\F(X)$ contains $\E$. Then $\mathcal{X}$ has a unique maximal element,
denoted by $C_S(\E)$ and called the \emph{centralizer in} $S$ \emph{of} $\E$.
Moreover, $C_S(\E)$ is strongly $\F$-closed, and there is a normal subsystem
$\C_\F(\E)$ of $\F$ based on $\C_S(\E)$. 
\end{theorem}
 
Most of the time we will use this characterization of $C_S(\E)$, but for the
proof of Theorem~\ref{T:maingrp}, it is more natural to use Aschbacher's
initial description \cite[6.1]{AschbacherGeneralized} of $C_S(\E)$ as follows.

First we recall some terminology and some ideas from
\cite[Section~4]{AschbacherNormal}. A saturated fusion system $\F$ over $S$ is
\emph{constrained} if $C_S(O_p(\F)) \leq O_p(\F)$. If $\F$ is constrained, then
it has a \emph{model} $G$ \cite[III.5.10]{AschbacherKessarOliver2011}.  That
is, $G$ is a finite group with $S \in \Syl_p(G)$ and $C_G(O_p(G)) \leq O_p(G)$
such that $\F = \F_S(G)$. Any two models for $\F$ are isomorphic by an
isomorphism which is the identity on $S$; we refer to this as the strong
uniqueness of models.

Now let $\E$ a normal subsystem of $\F$ on $T$. Let $U \in \E^c \cap \F^f$.
Then $U \in \E^{fc}$ and $\E(U) := N_\E(U)$ is saturated and constrained system
on $N_T(U)$. Also $UC_S(U) \in \F^{fc}$, and so $\D(U) := N_\F(UC_S(U))$ is
saturated and constrained system on $N_S(U)$. Furthermore $\E(U) \norm \D(U)$.
Let $G(U)$ be a model for $\D(U)$, and $H(U)$ be the unique normal subgroup of
$G(U)$ which is a model for $\E(U)$ \cite[Theorem~1]{AschbacherNormal}. Then
$C_{N_S(U)}(H(U))$ is a well-defined subgroup of $S$ by strong uniqueness of
$G(U)$. Aschbacher defines 
\[ 
I = \bigcap_{U \in \E^c \cap \F^f} C_{N_S(U)}(H(U))
\]
and 
\[
C_S(\E) = \bigcap_{\phi \in \Aut_\F(TC_S(T))} I^\phi.
\]
This is motivated by the fact that $\E = \gen{\Aut_\E(U)^\phi \mid U \in \E^c
\cap \F^f,  \phi \in \Aut_\F(T)}$ \cite[Theorem~3]{AschbacherNormal} and
that the restriction map $\Aut_\F(TC_S(T)) \to \Aut_\F(T)$ is surjective.

The model version of the construction of centralizers makes clear the
connection between the centralizer of a normal subgroup in a finite group, and
the corresponding centralizer at the level of fusion systems. We record this
connection as follows.

\begin{lemma}\label{L:centgrp}
Let $G$ be a finite group and $N$ a normal subgroup of $G$. Let $S \in
\Syl_p(G)$, $T = S \cap N$, $\F = \F_S(G)$, and $\E = \F_T(N)$.  Let $\H(N) =
\{N_N(V) \mid  V \in \E^{fc}\}$. Then 
\begin{enumerate}[label=\textup{(\alph{*})}]
\item $C_S(\E)$ is the largest subgroup $Y$ of $S$ for which $[H, Y] \leq
O_{p'}(H)$ for every $H \in \H(N)$,
\item $C_S(N/O_{p'}(N)) \leq C_S(\E)$, and
\item $C_S(\E) = C_S(N/O_{p'}(N))$ if $\Aut(N/O_{p'}(N))$ contains
no element of order $p$ centralizing $H/O_{p'}(H)$ for every $H \in
\H(N/O_{p'}(N))$. (Here, $T$ and $\E$ are identified with
$\bar{T} := TO_{p'}(N)/O_{p'}(N)$ and $\F_{\bar{T}}(N/O_{p'}(N))$ via the
canonical isomorphisms.)
\end{enumerate}
\end{lemma}
\begin{proof}
For each $U \in \E^{fc}$, set $H_U = N_N(U)$ for brevity. Let $Y \leq S$ and
assume $[H_U, Y] \leq O_{p'}(H_U)$ for each such $U$. Note then that $Y$
centralizes $T$ because $[T, Y] \leq T \cap O_{p'}(H_T) = 1$.  Fix $U \in \E^c
\cap \F^f \subseteq \E^{fc}$ and set $G_U = N_G(UC_S(U))$. Then $G(U) :=
G_U/O_{p'}(G_U)$ is a model for $N_\F(UC_S(U))$ with Sylow $p$-subgroup $N_S(U)
\geq Y$.  Since $H_U \norm G_U$, we have $O_{p'}(H_U) = O_{p'}(G_U) \cap H_U$.
Hence $H(U) := H_U/O_{p'}(H_U)$ is a normal subgroup of $G(U)$ while being a
model for $N_\E(U)$, and so $Y \leq I$ with $I$ as above. But
$\Aut_\F(TC_S(T))$ acts on $\E^{fc}$, and so $Y \leq C_S(\E)$. 

Now set $Y = C_S(\E)$. Then $Y$ is strongly $\F$-closed, and $[H_V, Y] \leq
O_{p'}(H_V)$ whenever $V \in (\E^c \cap \F^f)^{N_G(TC_S(T))}$ by reversing the
argument in the last paragraph. Let $U \in \E^{fc}$. By a Frattini argument,
there is $V \in (\E^c \cap \F^f)^{N_G(TC_S(T))} \cap U^\E$. Fix such a $V$,
and let $h \in N$ with $V^h = U$. By Theorem~\ref{T:C_S(E)}, there is $g \in
C_G(Y)$ and $t \in C_G(V)$ with $h = tg$. Hence $V^g = U$ and
$[H_U, Y] = [H_V^g, Y^g] = [H_V, Y]^g \leq O_{p'}(H_V)^g
= O_{p'}(H_U)$ as required. This completes the proof of (a). 

Let $X = C_S(N/O_{p'}(N))$. Then $[T, X] \leq T \cap O_{p'}(N) = 1$.  Also,
$[H, X] \leq H \cap O_{p'}(N) \leq O_{p'}(H)$ for any subgroup $H$ of $N$
normalized by $X$, so (b) is immediate from (a). 

For (c), assume that $X < Y$. Then $Y/X$ acts on $N/O_{p'}(N)$ and centralizes
(the image of) $H/O_{p'}(H)$ for each $H \in \H(N)$. Since $O_{p'}(N) \cap H
\leq O_{p'}(H)$, the same is true for each $\bar{H} \in \H(N/O_{p'}(N))$. Now
(c) is clear.
\end{proof}

We will need a lemma examining in a special case how centralizers behave under
quotienting by a strongly closed subgroup. First, the following lemma shows
that for the purposes of computing the centralizer in $S$ of a normal subsystem
$\E$ of $\F$, we may restrict to the subsystem $\E S$ of
Theorem~\ref{T:Aschppower}.

\begin{lemma}\label{L:SEsufficesforC_S(E)}
Suppose $\F$ is a saturated fusion system on $S$ and $\E$ is a normal subsystem
of $\F$ on $T$. Let $Q \leq C_S(T)$. Then $C_{\F}(Q) \geq \E$ if and only if
$C_{\E S}(Q) \geq \E$.
\end{lemma}
\begin{proof}
Let $\phi \in \Hom_\E(U,V)$ for subgroups $U$ and $V$ of $T$. If $\phi$ lies in
$C_{\E S}(Q)$, then it clearly lies in $C_\F(Q)$. Suppose $\phi \in C_\F(Q)$.
Then $\phi$ extends to a morphism $\tilde{\phi} \in \Hom_\F(QU, QV)$ with
$\tilde{\phi}|_Q = \id_Q$, and it suffices to show that $\tilde{\phi} \in \E
S$.  By Alperin's fusion theorem applied in $\E$, it is enough to show this
when $U=V$ and $U \in \E^{fc}$, and we may take $\phi$ and $\tilde{\phi}$ to be
of $p^\prime$-order in this case. Then $\tilde{\phi} \in A^\circ(QU)$ in the
sense of \cite[Definition~1]{Henke2013} and $\E S$ is generated by such
automorphism groups. 
\end{proof}

\begin{lemma}\label{P:centralizermodOp}
Let $\F$ be a saturated fusion system on the $p$-group $S$. Suppose that $\E$
is a normal subsystem of $\F$ on the strongly $\F$-closed subgroup $T$ of $S$.
Set $Q = C_S(\E)$, and assume that $Q \cap T = 1$. Then $C_{S/Q}(\E S/Q) = 1$. 
\end{lemma}
\begin{proof}
By Lemma \ref{L:SEsufficesforC_S(E)}, we may assume that $\F = \E S$. By
Theorem~\ref{T:C_S(E)}, $Q$ is strongly $\F$-closed.  Let $\theta\colon \F \to
\F/Q$ be the surjective morphism of fusion systems and denote passage to the
quotient by bars. Since $Q$ and $T$ are normal in $S$, we have $[Q, T] \leq Q
\cap T = 1$.  Also, for every $U$, $V \leq T$, 

\begin{align}\label{E:thetainducesbij}
\theta_{U,V}\colon \Hom_{\F}(U,V) \overset{\simeq}{\longrightarrow}
\Hom_{\bar{\F}}(\bar{U}, \bar{V}).
\end{align}
is a bijection since $Q \cap T = 1$.

Suppose the proposition is false and let $\bar{Q}_1 = C_{\bar{S}}(\bar{\E})
\neq 1$, where $Q_1 > Q$ is the preimage of $\bar{Q}_1$ under $\theta$.  By
Theorem~\ref{T:C_S(E)}, $\bar{Q}_1$ is strongly $\bar{\F}$-closed. By
\cite[8.9.1]{AschbacherNormal}, $Q_1$ is strongly $\F$-closed.
Since $T \norm S$, we have $[Q_1, T] \leq T$. But $[Q_1, T] \leq Q$ as well
because $[\bar{Q}_1,\bar{T}]=1$.  Therefore, $[Q_1, T] \leq Q \cap T = 1$.

We will show that $\E \leq C_{\F}(Q_1)$, supplying a contradiction.  Let $U$ be
a fully $\E$-normalized, $\E$-centric subgroup of $T$ and let $\varphi \in
\Aut_{\E}(U)$. Since $U$ is fully $\E$-normalized, we have that $\varphi =
\varphi'c_t$ for some $\varphi' \in O^p(\Aut_\E(U))$ and $t \in T$. If
$\varphi' \in C_{\F}(Q_1)$, then so is $\varphi$ since $Q_1$ centralizes $T$.
Hence we may assume that $\varphi$ has order prime to $p$. 

Let $\bar{\varphi} = (\varphi)\theta$. Then by definition of $\bar{Q}_1$,
$\bar{\varphi}$ extends to $\bar{\varphi}_1 \in
\Aut_{\bar{\F}}(\bar{Q}_1\bar{U})$ such that
$\resbar{\bar{\varphi}_1}{\bar{Q}_1} = \id_{\bar{Q}_1}$. Let $\phi_1$ be a
morphism in $\Aut_\F(Q_1U)$ such that $(\phi_1)\theta = \bar{\varphi}_1$.  Then
by (\ref{E:thetainducesbij}), $\phi_1$ restricts to $\varphi$ on $U$.  Let
$\tilde{\varphi}$ be the $p'$-part of $\phi_1$. As $\varphi$ has order prime to
$p$, $\tilde{\varphi}$ still restricts to $\varphi$ on $U$.  Furthermore,
$\resbar{\phi_1}{Q_1}$ stabilizes the series $1 \leq Q \leq Q_1$ and
centralizes $Q_1/Q$, so the same is true for $\tilde{\varphi}$. Hence as
$O^p(\F) = O^p(\E)$, we have $[Q_1, \tilde{\phi}] = [Q, \tilde{\phi}] \leq T
\cap Q = 1$, so $\tilde{\phi}$ restricts to the identity on $Q_1$ and extends
$\phi$.  By Alperin's fusion theorem, $\E \leq C_\F(Q_1)$ contradicting the
maximality of $Q$.
\end{proof}

\subsection{Thompson-Lyons transfer lemma}\label{SS:transfer}
The main technical tool in the proof of Theorem~\ref{T:main} is the
Thompson-Lyons transfer lemma, proved in \cite{LyndTL}. 

\begin{theorem}
\label{T:TT}
Let $\F$ be a saturated fusion system over a $2$-group $S$ with $\F = O^2(\F)$.
Suppose $T$ is a proper normal subgroup of $S$ with $S/T$ abelian. Let $u \in
S-T$ of least order. If the set of cosets of $T$ containing a fully
$\F$-centralized conjugate of $u$ is linearly independent in $\Omega_1(S/T)$,
then $u$ has a fully $\F$-centralized $\F$-conjugate in $T$. 
\end{theorem}

The reader will notice that the linear independence condition in
Theorem~\ref{T:TT} holds automatically in the case that $S/T$ is cyclic.

Since Theorem~\ref{T:TT} is used so often, we illustrate how it can be applied
together with Alperin's fusion theorem to describe the fusion systems $\F =
O^2(\F)$ on nonabelian $2$-groups of maximal class.
\begin{lemma}\label{L:fsmaxcl}
Let $\F$ be a saturated fusion system on the $2$-group $S$ with $\F = O^2(\F)$.
If $S$ is nonabelian of maximal class, then $\F$ is uniquely determined by $S$
up to isomorphism and one of the following holds.
\begin{enumerate}[label=\textup{(\alph{*})}]
\item $S \cong D_{2^k}$ with $k \geq 3$, and for any odd prime power
$q$ with $\nu_2(q^2-1) = k+1$, we have $\F \cong \F_S(G)$ with $G \cong L_2(q)$,
\item $S \cong SD_{2^k}$ with $k \geq 4$, and for any odd prime power $q
\equiv 3 \pmod 4$ with $\nu_2(q+1)=k-2$, we have $\F \cong \F_S(G)$ with $G
\cong L_3(q)$.
\item $S \cong Q_{2^k}$ with $k \geq 3$, and for any odd prime power $q$
with $\nu_2(q^2-1) = k$, we have $\F \cong \F_S(G)$ with $G \cong SL_2(q)$. 
\end{enumerate}
\end{lemma}
\begin{proof}
If $P$ is a subgroup of $S$, then $P$ is cyclic, dihedral, semidihedral, or
quaternion. Hence $P$ has automorphism group a $2$-group unless $P \cong C_2
\times C_2$ or $Q_8$. Hence, if $P$ is a proper $\F$-radical subgroup of $S$,
then $P \cong C_2 \times C_2$ or $Q_8$ with $\Aut_\F(P)$ isomorphic to $S_3$ or
$S_4$, respectively. 

Let $C = \gen{c}$ be the cyclic maximal subgroup of $S$, and let $Z(S) =
\gen{z} \leq C$.  Let $P \in \F^{cr}$ and suppose that $P \cong C_2 \times
C_2$. Let $u \in P - Z(S)$. Then $u$ lies outside the cyclic maximal subgroup
$T$ of $S$. By Theorem~\ref{T:TT}, $u$ is $\F$-conjugate into $T$, and
therefore $\F$-conjugate to $z$. By the above description of the members of
$\F^{cr}$, $P$ is the unique proper $\F$-centric and $\F$-radical subgroup
containing $u$.  Therefore, $u$ is in fact $\Aut_\F(P)$-conjugate to $z$ by a
morphism of order $3$, and $\Aut_\F(P) \cong S_3$. 

Let $P \in \F^{cr}$ with $P \cong Q_8$. Then $S$ is semidihedral or quaternion,
and $P = \gen{u,z_1}$ with $z_1 \in P \cap C$. If $S$ is semidihedral, let $T =
\Omega_1(S)$, the dihedral maximal subgroup of $S$. If $S$ is quaternion, let
$T = C$. In either case, $u$ is of least order outside $T$, and the only elements
of $T$ of order $4$ are $z_1$ and $z_1^{-1}$. Hence by Theorem~\ref{T:TT},
$u$ is $\F$-conjugate to $z_1$. As in the previous paragraph, it follows that
$\Aut_\F(P) \cong S_4$ unless $S=P$, in which case $\Aut_\F(P) \cong A_4$. 

We have determined the automorphism groups $\Aut_\F(P)$ for $P \in \F^{cr}$.
Therefore, $\F$ is uniquely determined by Alperin's fusion theorem, and as is
described in (a)-(c).
\end{proof}

\subsection{Tameness and $p$-power extensions}\label{S:tame}

In this subsection we address the problem of determining the structure of
extensions of fusion systems, which is resolved via recent work of Andersen,
Oliver, and Ventura \cite{AOV2012}. Under the conditions of
Theorem~\ref{T:main}, this problem manifests itself in the determination of an
involution centralizer from the description of its generalized Fitting
subsystem. For suppose given a saturated fusion system on a $2$-group and
assume $\C$ is an involution centralizer on the subgroup $T$ with $F^*(\C) =
O_2(\C)\E$ where $\E$ is quasisimple. One is confronted with the possibility
that $\bar{\C} = \C/O_2(\C)$ is exotic extension of $\bar{\E} = \E/Z(\E)$ even
when the latter is realizable by a simple group.

Let $\F$ be a saturated fusion system on $S$ and $\F_0 = O^p(\F)$ on $S_0$.
Assume $O_p(\F_0) = 1$.  Theorem~A of \cite{AOV2012} says that if $\F_0$ is
strongly tamely realizable by $G_0$ (with $O_{p'}(G) =1$), then $\F$ is
realizable as well. Here we indicate how to follow the proof of that theorem to
obtain the additional information that $G$ may be chosen so that $\F = \F_S(G)$
and $\Inn(G_0) \leq G \leq \Aut(G_0)$, provided $C_S(\F_0) = 1$. 

Roughly, $\F_0$ is \emph{tame} if $\F_0$ is realizable by a finite group $G_0$
such that every outer automorphism of the canonical linking system of $G_0$ is
induced by an outer automorphism of $G_0$. Then $\F_0$ is said to be tamely
realizable by $G_0$. An example of a fusion system $\F_0$ which is tame,
realizable by $G_0$, but not tamely realizable by $G_0$ is obtained with $G_0 =
A_7$, which is missing the ``diagonal automorphism'' present on $\F_2(L_2(q))$,
($q \equiv \pm 7 \pmod{16}$). See Proposition~\ref{P:dihtame}.

The fusion system $\F_0$ is \emph{strongly tame} if, in addition, certain
higher limits of functors associated to $G_0$ vanish; this is expressed by
saying $G_0$ lies in a certain class $\mathfrak{G}(p)$ of finite groups. We
refer the reader to \cite[II.\S~3,4]{AschbacherKessarOliver2011} for background
on linking systems.  We also point the reader to \cite{AOV2012} for details on
tameness and the precise meaning of $\mathfrak{G}(p)$.

\begin{proposition}\label{P:O^ptame}
Let $\F$ be a saturated fusion system on the $p$-group $S$. Suppose that $\F_0
= O^p(\F)$, $C_S(\F_0) = 1$, and that $\F_0$ is strongly tamely realized by
$G_0$ with $O_{p'}(G_0) = 1$. Then $\F$ is realized by a finite group $G$ such
that $\Inn(G_0) \leq G \leq \Aut(G_0)$ and $O^p(G) \leq \Inn(G_0)$.
\end{proposition}
\begin{proof}
A tame fusion system can always be realized by a finite group with no
nontrivial normal $p'$ subgroups by \cite[Lemma~2.19]{AOV2012}, and if the
fusion system is strongly tame, then the group can be chosen to lie in
$\mathfrak{G}(p)$ as well. Since $G_0 \in \mathfrak{G}(p)$, $\F_0$ has a unique
centric linking system $\L_0$.  In addition, there is a unique centric linking
system $\L$ associated to $\F$ by \cite[Proposition~2.12(a)]{AOV2012}. 

Next observe that $Z(\F) = Z(\F_0) = 1$, since $C_S(\F_0) = 1$.  As
$O_{p'}(G_0) = 1$, any central subgroup of $G_0$ must be a $p$-subgroup of $S$
central in $\F_0$, so it follows that $Z(G_0) = Z(\F_0) = 1$.  Now
\cite[Proposition~1.31(a)]{AOV2012} and \cite[Proposition~2.16]{AOV2012} apply
together to give that $\F$ is tamely realized by a finite group $G$ such that
$G_0 \norm G$ and with $G/G_0 \cong \L/\L_0$, a $p$-group.  Now $C_S(G_0) = 1$
again follows from $C_S(\F_0) = 1$. As $O_{p'}(G_0) = 1$ and $G/G_0$ is a
$p$-group, $O_{p'}(G) = 1$ giving $C_G(G_0) = 1$. Thus $G$ is isomorphic to a
subgroup of $\Aut(G_0)$ containing $\Inn(G_0) \cong G_0$. 
\end{proof}

We will apply Proposition~\ref{P:O^ptame} in the case where $\F_0 \cong
\F_2(L_2(q))$ for appropriate $q$. Hence we need to know

\begin{proposition}\label{P:l2qinL2}
Let $K \cong L_2(q)$ for $q \equiv \pm 1 \pmod{8}$.  Then $K \in
\mathfrak{G}(2)$. 
\end{proposition}
\begin{proof}
It is shown in \cite[Proposition~7.5]{Oliver2006} that classical groups in odd
characteristic lie in $\mathfrak{G}(2)$. Alternatively, Propositions~4.2 and
4.6(b) of \cite{Oliver2006} show that any finite simple group of $2$-rank at
most $3$ lies in $\mathfrak{G}(2)$ from general considerations.
\end{proof}

\section{Structure of the involution centralizer}\label{S:centralizer}

In this section we lay the groundwork for the study of $\F$ as in
Theorem~\ref{T:main} by studying some consequences of
Proposition~\ref{P:O^ptame}, and fixing notation. 
First we record some information about $\Aut(L_2(q))$.
\begin{lemma}\label{L:l2qomnibus}
Fix an odd prime power $q$. Let $K$ be a finite group isomorphic to $L_2(q)$
and $P \in \Syl_2(K)$. Let $H = \Aut(K)$, and $T \in \Syl_2(H)$. Then
\begin{enumerate}[label=\textup{(\alph{*})}]
\item $H = K\gen{h}F$ with $h \in \mathcal{I}_2(T)$, $F$ is cyclic, and
$\Out(K) \cong \gen{h} \times F$, 
\item $P$ is dihedral of order $2^k$ where $\nu_2(q^2-1) = k+1$, 
\item $K\gen{h} \cong PGL_2(q)$ and $P\gen{h}$ is dihedral,
\item $F$ is isomorphic to the Galois group of $GF(q)$ inducing field
automorphisms on $K$,
\item if $F_T := T \cap F$, then $P\gen{h} \cap F_T = 1$, and
$|F_T| \leq 2^{k-2}$, 
\item all involutions of $P$ are $K$-conjugate as are all involutions
of $Ph$,
\item if $f$ is an involution in $F_T := T \cap F$, then $P\gen{f} = P
\times \gen{f}$, $q$ is a square, $C_K(f) \cong PGL_2(q^{1/2})$, and $P
\cap O^2(C_K(f))$ is dihedral of order $2^{k-1}$,  
\item if $f$ is an involution in $F_T := T \cap F$ and $\gen{z} = Z(P)$,
then $[f,h] = (fh)^2 = z$, $P\gen{fh}$ is semidihedral, and all involutions
of $T$ lie in $P \cup Ph \cup Pf$, and
\item no involution of $H$ centralizes $N_K(P)/O_{2'}(N_K(P))$ and
$N_K(V)/O_{2'}(N_K(V))$ for each four subgroup $V$ of $P$.
\end{enumerate}
\end{lemma}
\begin{proof}
For (a)-(h), see Chapter 10, Lemma~1.2 of \cite{GLS6}. Point (i) is well-known
and follows from (f)-(h).  Such an involution would have to lie in $C_T(P) =
\gen{f, z}$ where $f \in \I_2(F)$ and $z \in Z(P)$. By point (f) for $z$ and
points (f)-(h) for $f$ and $fz$, each involution in $C_T(P)$ is a noncentral
element in $N_H(V)$ for some four subgroup $V$ of $P$. 
\end{proof}

We are now in a position to prove Theorem~\ref{T:maingrp} assuming
Theorem~\ref{T:main}.
\begin{proof}[Proof of Theorem~\ref{T:maingrp}]
Let $\F = \F_S(G)$. We may assume $\gen{x}$ is fully $\F$-centralized by
choosing $S$ appropriately. Then $T := C_S(x)$ is a Sylow $2$-subgroup of $C =
C_G(x)$.  Set $P = T \cap K \in \Syl_2(K)$, a dihedral group of order $2^k$
with $k \geq 3$, and set $\K = \F_P(K)$. Put $Q = C_T(\K)$, $\bar{K} =
K/O(K)$, and $Q_0 = C_T(\bar{K})$. Then $Q_0$ is a Sylow $2$-subgroup of
$C_C(\bar{K})$, which is cyclic by hypothesis. 

We check the hypotheses of Theorem~\ref{T:main} for $\F$. First we note that
\begin{eqnarray}\label{E:Q=Q0cyclic}
    Q = Q_0 \mbox{ is cyclic }
\end{eqnarray}
by Lemma~\ref{L:centgrp}(b,c) and Lemma~\ref{L:l2qomnibus}(i) in the case $K
\cong L_2(q)$. A similar (but easier) argument establishes this claim in case
$K \cong A_7$.  Now \eqref{E:Q=Q0cyclic} implies (2) of Theorem~\ref{T:main} is
satisfied.

Next, with $\C = C_\F(x)$,
\begin{eqnarray}\label{E:QnormC}
Q = O_2(\C).
\end{eqnarray}
By \eqref{E:Q=Q0cyclic}, $C_\C(\K)$ is the fusion system of the $2$-group $Q$.
By \cite[9.8.2]{AschbacherGeneralized}, $\K = E(\C)$. Now
\cite[9.12]{AschbacherGeneralized} applies to give that $O_2(\C) =
F^*(C_\C(\K)) = Q$, as desired.

Next, we show
\begin{eqnarray}\label{E:Y=1}
O_2(\F) = 1.
\end{eqnarray}
Let $Y = O_2(\F)$ and suppose that $Y \neq 1$. Then $\Omega_1(Z(Y)) \norm \F$.
Hence $\Omega_1(Z(Y)) \cap T$ is (nontrivial) normal in $\C$; this follows for
instance from the Aschbacher-Stancu characterization of normality
\cite[Proposition~4.62, Theorem~5.29]{CravenTheory}. So $1 \neq \Omega_1(Z(Y))
\cap T \leq \Omega_1(Q) = \gen{x}$. Hence $N_S(T) \leq N_S(\Omega_1(Z(Y)) \cap
T) = T$, from which it follows that $S=T$.  Therefore $\Omega_1(Z(Y)) =
\gen{x}$ is strongly $\F$-closed.  This contradicts the $Z^*$-theorem and hence
$Y=1$.

As $\F = O^2(\F)$ by Puig's hyperfocal subgroup theorem \cite{Puig2000}, and
$\Baum(S) \leq C_S(x) =  T$ by assumption, $\F$ satisfies the
hypotheses of Theorem~\ref{T:main}.  Hence $S \cong D_{2^k}\wr C_2$ and $\F
\cong \F_2(L_4(q_1))$ for some $q_1 \equiv 3 \pmod{4}$. Now a theorem of David
Mason \cite[Theorems~1.1, 3.14]{Mason1973} shows that $O^{2'}(G/O(G))$ is
isomorphic with $L_4(q_2)$ for some $q_2 \equiv 3 \pmod{4}$ or with $U_4(q_3)$
for some $q_3 \equiv 1 \pmod{4}$. From the structure of involution centralizers
in $L_4(q_2)$ and $U_4(q_3)$ (see \cite[6.5.2, 6.5.15]{Suzuki1986}) it follows
that $q_2^2 = q$ and $q_3^2 = q$ in the respective cases. This completes the proof
of Theorem~\ref{T:maingrp}.
\end{proof}

We now begin work on the proof of Theorem~\ref{T:main}. The following
hypothesis simply extracts those of Theorem~\ref{T:main} and fixes some
notation. Recall that a perfect fusion system is, by definition, one which
is equal to its hyperfocal subsystem.

\begin{hypothesis}\label{H:main}
Suppose that $\F$ is a saturated fusion system on the $2$-group $S$ with $\F =
O^2(\F)$ and $O_2(\F) = 1$.  Assume $x$ is a fully $\F$-centralized involution
in $S$ with $\Baum(S) \leq C_S(x)$. Write $T = C_S(x)$ and $\C = C_\F(x)$, and
assume that $\K$ is a perfect normal subsystem of $\C$ based on a dihedral
group $P$ of order $2^k$.  Assume $Q = C_T(\K)$ is cyclic. Set $R = QP = Q
\times P$.
\end{hypothesis}

Unless otherwise specified, we assume for the remainder of this paper that $\F$
is a fusion system satisyfing Hypothesis~\ref{H:main}, and we adopt the notation
there. The next proposition allows us to choose a suitable realization of $\K$. 

\begin{proposition}\label{P:dihtame}
For fixed $k$, the subsystem $\K$ is uniquely determined up to isomorphism. Fix
a prime $p \equiv 5 \pmod{8}$, let $q = p^{2^{k-2}}$, and set $K = L_2(q)$.
Then $\K$ is strongly tamely realized by $K$.
\end{proposition}
\begin{proof}
Lemma~\ref{L:fsmaxcl}(a) shows there is a unique fusion system $\K$ on
$P$ with $\K = O^2(\K)$.  That $\K$ is tamely realized by $K$ is the content of
\cite[Proposition~4.3]{AOV2012}. Strong tameness follows from
Proposition~\ref{P:l2qinL2}.
\end{proof}

For the remainder, we fix an odd prime power $q \equiv \pm 1 \pmod{8}$ with
$\nu_2(q^2-1) = k+1$, and we fix $K$ strongly tamely realizing $\K$ as in
Proposition~\ref{P:dihtame}.

Since $\K$ is a normal subsystem of $\C$, for each $T_1 \leq T$ with $P \leq
T_1$ we may form the product $\K T_1$ as in Theorem~\ref{T:Aschppower}. Then
$O^2(\K T_1) = O^2(\K) = \K$.
\begin{proposition}\label{P:TK/Q}
For each $T_1 \leq S$ with $R \leq T_1$, the quotient $\K T_1/Q$ is isomorphic
to the $2$-fusion system of a subgroup of $\Aut(K)$ containing $K$. 
\end{proposition}
\begin{proof}
We verify the hypotheses of Proposition~\ref{P:O^ptame} with $\F = \K T_1/Q$
and $\F_0 = \K Q/Q \cong \K$.  Denote quotients by $Q$ with bars. By
Proposition~\ref{P:dihtame}, $\K$ is strongly tamely realized by $K$ and
$O_{2'}(K) = 1$.  By Lemma~\ref{L:O^pbasic}, we have that $\bar{\K} =
O^2(\bar{\K T_1})$.  Since $Q \cap P = 1$, $C_{\bar{T}_1}(\bar{\K T_1}) = 1$ by
Lemma \ref{P:centralizermodOp}. Therefore $\K T_1/Q$ is the $2$-fusion system
of a subgroup of $\Aut(K)$ containing $\Inn(K) \cong K$.
\end{proof}

\begin{lemma}\label{L:felements}
Let $f \in T$ be an involution. If $f \in C_T(P)$ but $f \nin R$, then
\begin{enumerate}[label=\textup{(\alph{*})}]
\item $\K P\gen{f}$ is the fusion system of $K\gen{f}$ where
$f$ is an involutory field automorphism of $K$, and
\item if $V$ is a four subgroup of $P$, there exists a unique $i \in
\set{0,1}$ such that $C_\C(fz^i)$ contains $\Aut_{\K}(V)$. In this case, if
$V'$ is another four subgroup of $P$ not $P$-conjugate to $V$, then
$C_\C(fz^{1-i})$ contains $\Aut_{\K}(V')$.
\end{enumerate}
\end{lemma}
\begin{proof}
By Theorem~\ref{T:Aschppower}(c) the set of saturated subsystems of $\K T$ is
in one-to-one correspondence with the set of subgroups of $T$ containing $P$
via the bijection $X \mapsto \K X$. Factoring by $Q$ induces an isomorphism of
$\K P\gen{f}$ with the fusion system of an extension of $K$ (by
Proposition~\ref{P:TK/Q}) containing an involution outside $K$ centralizing a
Sylow $2$-subgroup of $K$. This is unique and the required extension by
Lemma~\ref{L:l2qomnibus}, proving (a).

Now by Lemma~\ref{L:l2qomnibus}(g), $C_{\K P\gen{f}}(f)$ contains $\Aut_\K(U)$
for each $U$ in some unique $P$-class of four-subgroups in $P$. Furthermore,
there is an (abstract) $\K P\gen{f}$-fusion preserving isomorphism $c_h$ of
$P\gen{f}$ which swaps the $P$-classes of four-subgroups of $P$ and interchanges
$f$ and $fz$ by Lemma~\ref{L:l2qomnibus}(c,h). So either $\Aut_{\K}(V)$ or
$\Aut_{\K}(V')$ is contained in $C_{\K P\gen{f}}(f)$, and the other $\K$-automorphism
group is contained in $C_{\K P\gen{f}}(fz)$. 
\end{proof}

We write $\K\gen{f}$ in place of $\K P\gen{f}$, for brevity.  In view of
Lemma~\ref{L:felements}, we also make the following definition in the situation
of Hypothesis~\ref{H:main}. 
\begin{definition}\label{D:felement}
We say that an involution $f \in T$ is an $f$-\emph{element} on $\K$ if $f \in
C_T(P)$ but $f \nin R$.
\end{definition}

Viewing $\bar{T} = T/Q$ as a subgroup of the Sylow $2$-subgroup of $\Aut(K)$
via Proposition~\ref{P:TK/Q}, let $\bar{F}$ denote the intersection of
$\bar{T}$ with group of field automorphisms of $\Aut(K)$, and let $F$ be the
preimage in $T$ of $\bar{F}$.  Also, let $F_1$ be the preimage in $T$ of
$\Omega_1(\bar{F})$. Hence, $F_1 \cap R = Q$, $|F_1\colon Q| = 1$ or $2$, and
$[F_1, P] \leq Q \cap P = 1$ by Lemma~\ref{L:l2qomnibus}(g). So $C_T(P) =
F_1Z(P)$ by Lemma~\ref{L:l2qomnibus}(c). 

Let $R_d$ denote the largest subgroup of $T$ containing $R$ for which $R_d/Q$
is dihedral. Thus $R_d$ contains $R$ with index $1$ or $2$.
Proposition~\ref{P:TK/Q} and Lemma~\ref{L:l2qomnibus} give that $\K R_d/Q$ is
the $2$-fusion system of $L_2(q)$ or $PGL_2(q)$, respectively.

\begin{lemma}\label{L:2rank3or4}
Assume Hypothesis~\ref{H:main}. If $T$ contains an $f$-element on $\K$, then
$T$ \textup{(}and hence $S$\textup{)} has $2$-rank $4$ and $J(S) = J(RF_1)$.
Otherwise $T$ \textup{(}and hence $S$\textup{)} has $2$-rank $3$.
\end{lemma}
\begin{proof}
As remarked previously, $C_T(P) = F_1Z(P)$. Thus, there is an $f$-element on
$\K$ in $T$ if and only if $F_1 \neq Q$ and $F_1$ splits over $Q$.  If $h_1 \in
\I_2(T-RF)$, then $C_{R}(h_1)$ is of $2$-rank $2$ by
Lemma~\ref{L:l2qomnibus}(c).  Lemma~\ref{L:l2qomnibus}(h) shows that $h_1$
centralizes no involution in $RF-R$.  Hence, $C_T(h_1)$ is of $2$-rank $3$.
Suppose $f$ is an $f$-element of $T$. Then as $f \in C_T(P)$ the $2$-rank of
$RF$ is $4$. Hence, $J(T) = J(RF) = J(RF_1)$ in this case, and $T$ is of
$2$-rank $4$.  If $T$ contains no $f$-element, then $J(RF) = J(R)$ is of
$2$-rank $3$, and so $T$ is of $2$-rank $3$ as well.  Since $J(S) \leq \Baum(S)
\leq T$, we have $J(S) = J(T)$, and the $2$-rank of $S$ is the $2$-rank of $T$.
\end{proof}

\section{The $2$-central case}\label{S:2central}

We begin now the heart of the analysis of a fusion system $\F$ satisfying
Hypothesis~\ref{H:main}.  The objective of the current section is to handle the
case in which $x$ lies in the center of $S$, i.e. in which $S = T$. We assume
$S = T$ throughout this section.  Eventually, in
Proposition~\ref{P:2centralfinal}, we will reach the conclusion that there is
no such $\F$. This section and the next are modeled on the treatment in
\cite{GLS6}, in particular Proposition~3.4 of Chapter~2 and Section~12 of
Chapter~3 there.

Adopt the notation of Section~\ref{S:centralizer} and in particular of
Hypothesis~\ref{H:main}.  Thus $\C = C_\F(x)$, $\K$ is a component of $\C$ on
the dihedral group $P$, $T = C_S(x)$ and $Q = C_T(\K)$. Recall the definition
of $f$-element from Definition~\ref{D:felement} and the definitions of $R_d$
and $F$.  Also set $Z(P) = \gen{z}$, $T_0 = \Omega_1(T)$, and $Z =
\Omega_1(Z(T_0))$.  Directly from the definition of the centralizer and the
fact that $\K$ has a single class of involutions, we have
\begin{align}
\text{all involutions of $P^\#x$ are $\C$-conjugate}
\label{E:allPxconjugate}
\end{align}
We begin with two lemmas which apply throughout this section, after which we
state the main technical result of the present case.
\begin{lemma}\label{L:zweaklyclosedinZ0}
$z$ is weakly $\F$-closed in $Z$. 
\end{lemma}
\begin{proof}
As $T_0 = \Omega_1(T)$ is weakly $\F$-closed in $T$, $\Aut_\F(T_0)$ controls
fusion in $Z = \Omega_1Z(T_0)$ by Lemma~\ref{L:BurnsideLemma}.  Suppose $z$ is
not weakly $\F$-closed in $Z$, and let $\phi \in \Aut_\F(T_0)$ such that $z
\neq z^\phi \in Z$. We have $P \norm T$, so $z \in Z(T)$. Thus, we may take
$\phi$ to be of odd order. Set $P_i = P^{\phi^i}$ for $i \in {\bf Z}_{\geq 0}$.
As $P \norm T$ and $P_i = \Omega_1(P_i)$ for all $i$, we have $P_i \norm T_0$
for all $i$. Then
\begin{eqnarray}
Z(P_i) \cap Z(P_j) = 1 \mbox{ for each } i \neq j \in \set{0,1,2}
\label{E:centersintersection}
\end{eqnarray}
because $\phi$ has odd order and $z^\phi \neq z$. Furthermore, $Z(P_i) \leq
Z$ for all $i$.

Let $i \neq j \in \set{0,1,2}$, and suppose that $P_i \cap P_j \neq 1$. As $P_i
\cap P_j \norm P_j$, $Z(P_j) \leq P_i \cap P_j$. Hence $Z(P_j) \leq P_i \cap Z
= Z(P_i)$, contrary to \eqref{E:centersintersection}. As $[P_i, P_j] \leq P_i
\cap P_j$, we have $P_0P_1P_2 \cong P_0 \times P_1 \times P_2$. But $T$ has
$2$-rank at most $4$ by Lemma~\ref{L:2rank3or4}, so this is a contradiction. 
\end{proof}

\begin{lemma}\label{L:zweaklyclosedviaP}
If $\phi \in \Hom_\F(P,T)$, then $z^\phi = z$. 
\end{lemma}
\begin{proof}
First recall that $P \norm T$. Let $C$ be the cyclic maximal subgroup of $P$.
Then $C \norm T$. Now $[T,T] \leq R$ from Proposition~\ref{P:TK/Q}, and in fact
$[T,T] \leq C_R(C) = QC$ by Lemma~\ref{L:autD}(b). So $T/QC$ is abelian.
Thus 
\begin{align}
\Omega_1([T,T]) \leq \Omega_1(QC) = \gen{x,z} \leq Z.
\label{E:omega1comm}
\end{align}

Let $\phi \in \Hom_\F(P,T)$. Then $[u_1, u_2] = z$ for a pair of involutions
$u_1$ and $u_2$ of $P$, and so $z^\phi = [u_1^\phi,u_2^\phi] \leq
\Omega_1([T,T]) \leq Z$ by (\ref{E:omega1comm}). Now
Lemma~\ref{L:zweaklyclosedinZ0} shows that $z^\phi = z$.  
\end{proof}

\begin{proposition}\label{P:2central}
Suppose $\F$ satisfies Hypothesis~\ref{H:main} with $S = T$.  Then no
involution of $T$ is an $f$-element.
\end{proposition}

Assume the hypotheses and notation of the proposition, but that the statement
is false. Let $f$ be an involutory $f$-element in $T$. We proceed in a series
of lemmas.
\begin{lemma}\label{L:z^FcapRfinP}
$z^\F \cap R\gen{f} \cin P$.
\end{lemma}
\begin{proof}
We first show that 
\begin{eqnarray}
\label{E:zweaklyclosedCTP}
z^\F \cap C_T(P) = \gen{z}.
\end{eqnarray}
Let $y \in z^\F \cap C_T(P)$ and choose $\phi \in \F$ with $y^\phi = z$. Since
$z \in Z(T)$, $\phi$ extends to a morphism on $P \leq C_T(y)$. Therefore $y = z$ by
Lemma~\ref{L:zweaklyclosedviaP}, and (\ref{E:zweaklyclosedCTP}) holds.

Now we suppose the lemma fails, and we let $y \in z^{\F} \cap (R\gen{f} - P \cup
C_T(P))$ be arbitrary. We claim that 
\begin{eqnarray}
y \mbox{ is } \C\mbox{-conjugate to an element of } C_T(P).
\label{E:yCconj}
\end{eqnarray}
Together with (\ref{E:zweaklyclosedCTP}), \eqref{E:yCconj} implies that $y$ is
$\C$-conjugate to $z$. Since $P$ is strongly $\C$-closed, this will yield a
contradiction.  Recalling that $R\gen{f} = Q\gen{f} \times P$, write $y=uf_0v$
with $u \in Q$, $f_0 \in \gen{f}$, and $v \in P$.  Since $y$ is an involution
outside $P \cup C_T(P)$, we have $uf_0 \neq 1 = (uf_0)^2$, $v \nin Z(P)$, and
$v^2 = 1$. Let $V$ be the four subgroup of $P$ containing $v$.  Then
$\Aut_\K(V) \leq C_{\K\gen{f}}(f_0z^i)$ for some $i \in \set{0,1}$ by
Lemma~\ref{L:felements}, and so there exists a morphism $\phi \in
\Aut_{\K\gen{f}}(V\gen{f})$ with $(f_0v)^\phi = (f_0z^i)^\phi(z^iv)^\phi =
f_0z^i\cdot z \in \gen{f,z}$. As $Q \norm \C$, the morphism $\phi$ extends to 
$\tilde{\phi} \in \C$ fixing $Q$, and hence $y^{\tilde{\phi}} =
(uf_0v)^{\tilde{\phi}} \in Q\gen{f,z} = C_T(P)$ confirming (\ref{E:yCconj}).
\end{proof}

\begin{lemma}\label{L:Pweaklyclosed}
$P$ is weakly $\F$-closed.
\end{lemma}
\begin{proof}
Suppose not. Choose by Alperin's fusion theorem a fully $\F$-normalized
subgroup $D \leq T$ containing $P$ and an automorphism $\phi \in \Aut_\F(D)$
with $P^\phi \neq P$. Since $P$ is normal in $T$, we can choose such a $\phi$
of odd order. Set $P_i = P^{\phi^i}$ for each $i$. Thus the subgroups $P_0=P$,
$P_1$ and $P_2$ are pairwise distinct by choice of $\phi$, whereas $Z(P_i) =
Z(P_0) = \gen{z}$ for all $i$ by Lemma~\ref{L:zweaklyclosedviaP}.

Now we examine the images of the $P_i$ in $\bar{T} = T/Q$. Since $\bar{z} \neq
1$, we have that $\bar{P_i} \cong D_{2^k}$ for all $i$. Furthermore, $P_i =
\Omega_1(P_i) \leq \gen{z^{\F}}$ as all involutions of $P$ are $\F$-conjugate.
Thus there exists $h \in R_d - R$ squaring into $Q$, and
\[
\bar{P_i} \leq \bar{P}\gen{\bar{h}}
\]
for all $i$ by Lemma~\ref{L:z^FcapRfinP} and the fact
(Lemma~\ref{L:l2qomnibus}(h)) that there are no involutions in
$\bar{P}\bar{h}\bar{f}$. 

Suppose that $\bar{P} = \bar{P_i}$ for some $i$. Then $PQ \geq P_i$, and so
$\mathcal{I}_2(P_i) \cin R$. By Lemma~\ref{L:z^FcapRfinP} then, $P_i \leq
\gen{z^\F \cap P_i} \leq P$.  So $P = P_i$. This shows that $\bar{P_0} \neq
\bar{P_1}$ and $\bar{P_0} \neq \bar{P_2}$.  But $\bar{P}\gen{\bar{h}}$ is
dihedral and the $\bar{P_i}$ are among the two dihedral maximal subgroups of
$\bar{P}\gen{\bar{h}}$ so 
\begin{eqnarray} \label{E:P1=P2}
\bar{P_1} = \bar{P_2}.
\end{eqnarray}
Set $S_0 = P_0P_1$ and $S_1 = P_1P_2$, so that $S_0^\phi = S_1$. Then
$\bar{S_0} = \bar{P}\gen{\bar{h}} \cong D_{2^{k+1}}$ but 
\begin{eqnarray}\label{E:S1dih}
    \bar{S_1} \cong D_{2^k}
\end{eqnarray}
from (\ref{E:P1=P2}). 

As $\bar{P_0}$ and $\bar{P_1}$ are distinct dihedral maximal subgroups of
$\bar{P}\gen{\bar{h}}$, we have $[\bar{P_0}, \bar{P_1}]$ is the cyclic maximal
subgroup of $\bar{P_0}$.  So $[P_0, P_1]$ is the cyclic maximal subgroup of
$P_0$. But $[P_0, P_1] \leq P_1$ because the $P_i$ normalize each other, and
hence $[P_0, P_1]$ is the cyclic maximal subgroup of $P_1$ as well. It follows
that $P_0 \cap P_1$ has index $2$ in $P_0$ and $P_1$, and $|S_0| = 2|P_0| =
2^{k+1}$. Hence, $S_0 \cong \bar{S_0} \cong D_{2^{k+1}}$ and $S_1 = S_0^\phi$
is also isomorphic to $D_{2^{k+1}}$ with center $\gen{z}^\phi = \gen{z}$. As
$\bar{z} \neq 1$, $\bar{S_1} \cong D_{2^{k+1}}$. This contradicts (\ref{E:S1dih}) and
completes the proof.  
\end{proof}

\begin{lemma}\label{L:ifufcthenuzfc}
Let $u$ be an involution in $C_T(P)$. If $u$ is fully $\F$-centralized, then so
is $uz$.
\end{lemma}
\begin{proof}
Let $\phi \in \Hom_\F(\gen{uz}, T)$ with $(uz)^\phi$ fully $\F$-centralized.
Then $\phi$ extends to a morphism $\tilde{\phi}$ on $C_T(uz) = C_T(u) \geq P$,
and $z^{\tilde{\phi}} = z$ by Lemma~\ref{L:zweaklyclosedviaP}. Since $u$ is
fully $\F$-centralized, we have
\[
|C_T(uz)| = |C_T(u)| \geq |C_T(u^{\tilde{\phi}})| = |C_T(u^{\tilde{\phi}}z)| =
|C_T((uz)^{\tilde{\phi}})|
\]
and so $uz$ is fully $\F$-centralized as well. 
\end{proof}

\begin{lemma}\label{L:forfz}
There exists $f_0 \in f\gen{z}$ and $\phi \in N_\F(P)$ such that $f_0^\phi = x$.
\end{lemma}
\begin{proof}
We have $R \norm T$ with $T/R$ abelian. By Lemma~\ref{L:l2qomnibus}(h), there
are at most two cosets of $T-R$ containing an involution, so
Theorem~\ref{T:TT} applies.  Choose $\phi \in \Hom_\F(\gen{f}, R)$
with $f^\phi$ fully $\F$-centralized. Then $\phi$ extends to a morphism
$\tilde{\phi}$ on $C_T(f) \geq P$ normalizing $P$ by
Lemma~\ref{L:Pweaklyclosed}. Thus, $\tilde{\phi} \in N_\F(P)$ and so
$f^{\tilde{\phi}} \in \Omega_1(C_R(P)) - \gen{z} = x\gen{z}$ as $[P, f] = 1$.
Since one of $f^{\tilde{\phi}}$ or $(fz)^{\tilde{\phi}} = f^{\tilde{\phi}}z$
equals $x$, we are finished. 
\end{proof}

Now the next two lemmas give a contradiction, completing the proof of
Proposition~\ref{P:2central}.

\begin{lemma}\label{L:x^F=xz}
$xz \in x^\F$.
\end{lemma}
\begin{proof}
Suppose not. Recall $x$ is not $\F$-conjugate to $z$ by
Lemma~\ref{L:z^FcapRfinP}. So $x$ is not conjugate into $P$ as all involutions of $P$
are $\F$-conjugate. Further all involutions of $P^{\#}x$ are $\F$-conjugate
by \eqref{E:allPxconjugate}, so by assumption we have that 
\begin{eqnarray}\label{E:xwc}
\gen{x} \mbox{ is weakly } \F\mbox{-closed in } \gen{x} \times P.
\end{eqnarray}
Replacing $f$ by $fz$ if necessary, there exists a subgroup $D \leq T$
containing $\gen{x} \times P$ and a morphism $\phi \in \Hom_\F(D, T)$ with
$x^\phi = f$ and $P^\phi = P$ by Lemma~\ref{L:forfz}. By
Lemma~\ref{L:felements}(b), $\gen{f}$ is not weakly closed in $\K\gen{f}$ and
so there exists $\psi \in \Hom_{\F}(\gen{f}, \gen{f} \times P)$ with $f^\psi
\in Pf-\set{f}$. Then $x^{\phi\psi\phi^{-1}}
\in Px - \set{x}$, which contradicts (\ref{E:xwc}) and completes the proof.
\end{proof}

\begin{lemma}
$xz \nin x^\F$.
\end{lemma}
\begin{proof}
Let $\phi \in \F$ with $x^\phi = xz$. Then as $x$ and $z$ lie in $Z(T)$ we may
assume $\phi \in \Aut_\F(T)$ is of odd order by the extension and Sylow axioms.
Then $z^\phi = z$ by Lemma~\ref{L:Pweaklyclosed}, and so $\phi$ induces an
automorphism of $\gen{x,z}$ of order $2$, a contradiction.
\end{proof}

This completes the proof of Proposition~\ref{P:2central}. We can now prove the
main result of this section.

\begin{proposition}\label{P:2centralfinal}
Suppose $\F$ satisfies Hypothesis~\ref{H:main}. Then $T < S$. 
\end{proposition}
\begin{proof}
Suppose to the contrary that $T = S$.  From Proposition~\ref{P:2central}, there
is no involutory $f$-element in $T$, and hence there are no involutions in
$Rf$.  By Lemma~\ref{L:l2qomnibus}(h), it follows that 
\begin{eqnarray}
T_0 \leq R_d
\label{E:Omega1TinRd}
\end{eqnarray}
and hence $Z = \gen{x, z} \leq Z(T)$ in the present case.
Furthermore,
\begin{eqnarray}
T \mbox{ is of } 2\mbox{-rank } 3
\label{E:Tof2rank3}
\end{eqnarray}
by Lemma~\ref{L:2rank3or4}.

In view of Lemma~\ref{L:zweaklyclosedinZ0}, we know $z$ is not $\F$-conjugate
to $x$ or $xz$. Since fusion in $Z(T)$ is controlled in $\Aut_\F(T)$ and
$\Aut_T(T) \in \Syl_2(\Aut_\F(T))$,
\begin{eqnarray}
x, xz, \mbox{ and } z \mbox{ are pairwise not } \F\mbox{-conjugate}.
\label{E:Zisolated}
\end{eqnarray}
Our assumption that $O_2(\F) = 1$ yields that $x$ has an $\F$-conjugate
outside $Z$.  Apply Alperin's fusion theorem to obtain a fully
$\F$-normalized, $\F$-centric subgroup $D$ of $T$, and an automorphism $\alpha
\in \Aut_\F(D)$ with $x^\alpha \nin Z$. Set $h = x^\alpha$. 

Note that $h \nin R$, as otherwise $h$ would lie in $\Omega_1(R) = \gen{x}
\times P$. Because all involutions of $P \cup P^\#x$ are $\F$-conjugate to $xz$ or
$z$, (\ref{E:Zisolated}) would yield $h=x$, contrary to the choice of $h$.
Therefore, by (\ref{E:Omega1TinRd}):
\begin{eqnarray}
h \in R_d - R \mbox{ and } P\gen{h} \mbox{ is dihedral. }
\label{E:hinRd-R}
\end{eqnarray}
Since $D$ is $\F$-centric, it contains $Z =
\gen{x,z} \leq \Omega_1(Z(T))$, and hence $\Omega_1(Z(D)) = Z\gen{h}$ as $T$ is
of $2$-rank $3$. Set $A = Z\gen{h}$.  Then $h$ is $N_P(A)$-conjugate to $hz$ by
(\ref{E:hinRd-R}). If $h$ is $\Aut_\F(A)$-conjugate to $hx$ or $hxz$ then it is
$\Aut_\F(A)$-conjugate to both, so $x$ has exactly five conjugates under
$\Aut_\F(A)$ by (\ref{E:Zisolated}), which is not the case. So
\begin{eqnarray}
h^{\Aut_\F(A)} = \set{x, h, hz}.
\label{E:AutFAconjugatesofh}
\end{eqnarray}
Since $Q \norm T$ is cyclic, $Q\gen{h}$ is abelian, dihedral, semidihedral, or
modular.  But if $Q\gen{h}$ is nonabelian then $h$ is $N_Q(A)$-conjugate to
$hx$, contradicting \eqref{E:AutFAconjugatesofh}. So $[Q,h]=1$ and hence 
\begin{eqnarray}
\Omega_1(T) = \Omega_1(R_d) = \gen{x} \times P\gen{h}. 
\label{E:Omega1T}
\end{eqnarray}
We claim that
\begin{eqnarray}
P\gen{h} \mbox{ is normal in } T.
\label{E:P<h>normal}
\end{eqnarray}
If this does not hold, then from \eqref{E:Omega1T} and the fact that $P$ is
normal in $T$, there exists $t \in T$ with $h^t \in Phx$ (note: $h$ is not
$T$-conjugate into $\gen{x}P \norm T$). But all involutions in $Ph$ are
$P$-conjugate (\ref{E:hinRd-R}), and hence multiplying $t$ by a suitable
element of $P$, we have that $h$ is $N_T(A)$-conjugate to $hx$, contradicting
(\ref{E:AutFAconjugatesofh}).

We now complete the proof via transfer arguments. Let $F_1$ be the preimage of
$\Omega_1(F/Q)$ in $F$. Note that $F$ is cyclic or else $F_1$ is quaternion by
Proposition~\ref{P:2central}. If $F$ is cyclic, then as it covers $T/P\gen{h}$,
we can apply Theorem~\ref{T:TT} to get that $x$ is $\F$-conjugate
into $P\gen{h} \cap Z(T) =  \gen{z}$, an immediate contradiction to
(\ref{E:Zisolated}). So $F_1$ is quaternion.  Let $w \in F_1-Q$ of order $4$, so
that $w^2 = x$.  In the present situation, $R_d = Q \times P\gen{h}$, $w$ is of
least order in $T-R_d$ by (\ref{E:Omega1T}), and $T/R_d$ is cyclic.
So Theorem~\ref{T:TT} yields a morphism $\phi \in \F$ with
$w^\phi \in R_d$, and hence $(w^\phi)^2 \in Z$ by the structure of $R_d$.  As
$w^2 = x$, $x^\phi = x$ by \eqref{E:Zisolated}.  Thus $\phi \in \C$ and so
$\phi$ extends to a morphism $\tilde{\phi}$ on a subgroup of $T$ containing $Q$
because $Q \norm \C$. This forces $F_1^{\tilde{\phi}} \leq
\gen{Q,w^{\tilde{\phi}}}$ to be abelian, a final contradiction.  
\end{proof}

\section{The $2$-rank $3$ case}\label{S:2rank3}

Continuing the notation from Section~\ref{S:2central}, we prove here the
following reduction.
\begin{theorem}\label{T:2rank3}
Let $\F$ be a fusion system on $S$ satifying Hypothesis~\ref{H:main}. Then $T$
is of $2$-rank $4$. 
\end{theorem}
Throughout this section, assume to the contrary that $T$ is of $2$-rank $3$.
By Hypothesis~\ref{H:main}, $S$ is also of $2$-rank $3$.  From
Proposition~\ref{P:2centralfinal}, 
\begin{eqnarray}
    x \nin Z(S).
    \label{E:not2central}
\end{eqnarray}
Recall $F_1$ contains $Q$ with index $1$ or $2$ and $F_1/Q$ induces field
automorphisms on $K$. Then 
\begin{eqnarray}
C_T(P) = F_1 \times \gen{z}.
    \label{E:C_T(P)}
\end{eqnarray}

By Lemma~\ref{L:2rank3or4}, there exists no involution in $T$ which is an
$f$-element. Set $J = J(S) = J(T)$ for short. We have the inclusions $P \leq J
\leq \Omega_1(T) \leq R_d$.  This shows that $Z = \Omega_1(Z(\Omega_1(T))) \leq
\Omega_1(C_{T}(P)) = \Omega_1(F_1 \times \gen{z}) = \gen{x,z}$. So $Z =
\gen{x,z}$ and $Z$ coincides with $\Omega_1(Z(J))$.  Therefore, since $\Baum(S)
\leq T$ and (\ref{E:not2central}), we have that $T = \Baum(S)$ and
\begin{eqnarray}
    T \mbox{ is of index } 2 \mbox{ in } S.
    \label{E:|S:T|=2}
\end{eqnarray}
Fix $a \in S-T$. 

\begin{lemma}\label{L:x^a=xz}
$z \nin x^\F \cup (xz)^\F$.
\end{lemma}
\begin{proof}
If $z \in Z(S)$, then both $x$ and $xz$ are fully $\F$-centralized and not in
$Z(S)$ by \eqref{E:not2central}, so it suffices to show that $z \in Z(S)$.
Suppose that $z \neq z^a$. As $z^a \in Z = \gen{x,z}$, $P^a \neq P$. Since
$P$ and $P^a$ are normal in $T$, we have $[P^a, P] \leq P^a \cap P$ is normal
in both $P$ and $P^a$. Furthermore, $[P^a, P] \neq 1$ since otherwise $P^aP =
P^a \times P$ is of $2$-rank $4$. Therefore $[P^a, P]$ contains both $Z(P^a)$
and $Z(P)$.  But $[Z(P^a), P] \leq [Z, P] = 1$, forcing $Z(P^a) = Z(P)$
contrary to assumption.
\end{proof}

It follows in particular that $x^a = xz$.

\begin{lemma}\label{L:Fcyclic}
$Q = \gen{x}$ and $F$ is cyclic.
\end{lemma}
\begin{proof}
Suppose that $Q > \gen{x}$ and let $u \in Q$ with $u^2 = x$. Then $\gen{u^a}$
is a normal subgroup of $T$. Since $(u^a)^2 = xz$, we have $[\gen{u^a}, T] \leq
\gen{xz}$. But $[\gen{u^a}, P] \leq P$ as $P$ is normal, and it follows that
$u^a \in C_T(P) = F_1 \times \gen{z}$ by (\ref{E:C_T(P)}). Therefore, $xz
\in \mho^1(\gen{u^a}) \leq \mho^1(F_1 \times \gen{z}) \leq Q$, which is absurd.
So $Q = \gen{x}$. As $F/Q$ is cyclic, $F$ is abelian. Hence $F$ is cyclic by
Lemma~\ref{L:2rank3or4}.
\end{proof}

\begin{lemma}\label{L:aswaps}
Let $V$ be a four subgroup of $P$ and set $E = \gen{x} \times V$. Then $E^a$ is
not $T$-conjugate to $E$.
\end{lemma}
\begin{proof}
Suppose that $E^a$ is $T$-conjugate to $E$. Modifying $a$ if necessary, we may
assume that $a$ normalizes $E$. Now the subgroup $N = \gen{c_a, \Aut_\C(E)}$ of
$\Aut_\F(E)$ lies in $GL_3(2)$ and does not act transitively on $\mathcal{I}_2(E)$ by
Lemma~\ref{L:x^a=xz}. As $x^a = xz$, $N$ does not stabilize a point of $E$.
So $N$ must fix a line, which is then $V$. It follows that $N = \Aut_\F(E)
\cong S_4$.  Now $|\Aut_T(E)| = 2$, and we can obtain a contradiction to
(\ref{E:|S:T|=2}) by showing that $|\Aut_S(E)| = 8$, i.e. that $E$ is fully
automized in $\F$.

Suppose that $E$ is not fully $\F$-automized.  Either $J = \gen{x} \times P$ or
there exists an involution $h \in R_d - R$ and $J = \gen{x} \times P\gen{h}$.
In either case, there are exactly two $S$-classes of elementary abelian
subgroups of order $8$. Moreover, if $E_1 \in E^\F$ is fully $\F$-automized, then
$E_1 \nin E^S$, and so $\gen{E, E_1} = J$. By Alperin's fusion theorem, there
is a subgroup $D \in \F^{fc}$ and an automorphism $\alpha \in \Aut_\F(D)$ of
odd order such that $E_1 := E^\alpha$ is fully $\F$-automized. But then $J =
\gen{E, E_1} \leq D$, and consequently $\alpha$ restricts to a nontrivial (odd
order) automorphism of $J$. On the other hand, $\Aut(J)$ is a $2$-group by
Lemma~\ref{L:autDwr2}, a contradiction.
\end{proof}

\begin{lemma}\label{L:Rd>R}
$R_d > R$.
\end{lemma}
\begin{proof}
Suppose on the contrary that $R_d = R$. If $|T\colon PF| = 2$, then $J = \gen{x}
\times P$, and $T$ acts transitively on $\elem_{2^3}(T)$ contrary to
Lemma~\ref{L:aswaps}.  So $T = PF$. If $|F| > 2$, then $Z(T) = \Omega_2(F)
\times \gen{z}$ and so $\mho^1(Z(T)) = \gen{x}$ is normal in $S$, at odds with
(\ref{E:not2central}). So $T = R = J = \gen{x} \times P$. 

We now obtain a contradiction by a transfer argument. Note that as all
involutions of $Px$ are $\F$-conjugate to $x$, we have $P = \gen{z^\F \cap T}$
is normal in $S$ by Lemma~\ref{L:x^a=xz}. Moreover, the quotient $S/P$ is
abelian.  If $b \in x^\F$ is fully $\F$-centralized, then $C_S(b)$ has $2$-rank
$3$, whence $b \in T \cap x^\F \cin Px$. 
Theorem~\ref{T:TT} now says that $x$ is $\F$-conjugate into $P$,
contradicting Lemma~\ref{L:x^a=xz} and completing the proof.
\end{proof}

As a consequence of the previous lemma, $T$ is transitive on $\elem_{2^3}(R) =
\elem_{2^3}(PF)$.  Fix a four subgroup $V$ of $P$. By Lemma~\ref{L:aswaps} and
the preceding remark, $V^a \leq R_d$ but $V^a \nleq R$. Fix an involution $h \in V^a - P$.
Then $P_1 := P\gen{h}$ is dihedral of order $2|P|$, and therefore is generated
by $\F$-conjugates of $z$.  We have at this point that $J = \gen{x} \times P_1
= R_d$.  As each involution in $Phx$ is $P$-conjugate to $hx$, and hence
$S$-conjugate into $Px$, it follows from Lemma~\ref{L:x^a=xz} that 
\begin{align}
P_1 =  \gen{z^\F \cap P_1} = \gen{z^\F \cap T}. 
\label{E:Tcapz^F}
\end{align}
So 
\begin{eqnarray} P_1 \mbox{ is normal in } S.
\label{E:P1norm} 
\end{eqnarray}
We now compute $C_S(P_1)$. Recall that $C_T(P) = \Omega_2(F) \times \gen{z}$
from (\ref{E:C_T(P)}) and Lemma~\ref{L:Fcyclic}.  Moreover, if $\Omega_2(F) >
\gen{x}$, then $[\Omega_2(F), h] = \gen{z}$ by (\ref{E:P1norm}) and
Lemma~\ref{L:l2qomnibus}(h). Since $P_1$ centralizes $\Omega_1Z(J) = Z =
\gen{x, z}$, we have $C_S(P_1) \leq \Baum(S) \leq T$. Hence, $C_S(P_1) =
C_T(P_1) = Z$. Thus,
\begin{eqnarray}
S/P_1Z \mbox{ is abelian}
\label{E:S/P_1Z}
\end{eqnarray}
by Lemma~\ref{L:autD} because $P_1$ is nonabelian dihedral. 

Since $S/P_1$ has a cyclic normal subgroup $FP_1/P_1$, it is abelian,
quaternion, dihedral, semidihedral, or modular. However, by \eqref{E:S/P_1Z},
$S/P_1$ is abelian or modular, or else $|F| = 4$ and $S/P_1$ is dihedral or
quaternion of order $8$. We rule out each of these cases in turn.

\begin{lemma}\label{L:notabelian}
$S/P_1$ is not abelian.
\end{lemma}
\begin{proof}
Suppose $S/P_1$ is abelian. For any $b \in x^\F$ which is fully
$\F$-centralized, $C_S(b)$ is of $2$-rank $3$, and so $b \in J \cap x^\F \cin
P_1x$ by \eqref{E:Tcapz^F} and Lemma~\ref{L:x^a=xz}.
Now $x$ has an $\F$-conjugate in $P_1$ by Theorem~\ref{T:TT}, and this
contradicts Lemma~\ref{L:x^a=xz}.
\end{proof}

The next lemma shows that $S/P_1$ is not quaternion.
\begin{lemma}
There exists an involution $b$ in $S-T$. 
\end{lemma}
\begin{proof}
Let $b_1 \in S-T$. Modifying $b_1$ by an element of $F$, we may assume $b_1^2 \in P_1Z$
because $S/P_1$ is not cyclic (by the previous lemma). But then $P_1Z\gen{b_1}/Z$ is
dihedral or semidihedral because $b_1$ swaps the $P_1$-classes of four-subgroups of $P_1$
by Lemma~\ref{L:aswaps}. Modifying $b_1$ by an element of $P_1$ then, we may assume
$b_1^2 \in Z$. Now $C_Z(b_1) = \gen{z}$ so $b_1^2 \in \gen{z}$. Set $b = b_1$ if
$b_1$ is an involution, and set $b = xb_1$ otherwise. Then $b$ is an involution.
\end{proof}

Fix $b \in \I_2(S - T)$ guaranteed by the previous lemma.

\begin{lemma}
$S/P_1$ is not modular.
\end{lemma}
\begin{proof}
Suppose it is. Then $S/P_1$ has a unique four subgroup, covered by $\gen{x,b}$.
Hence, $S_0 := \Omega_1(S) = P_1Z\gen{b}$ and $S/S_0$ is cyclic. Let $w \in F$
with $w^2 = x$. Then $w$ is of least order outside $S_0$ and centralizes $FP$,
whence $|S\colon C_S(w)| \leq 4$.  Apply Theorem~\ref{T:TT} to obtain a
morphism $\phi$ in $\F$ with $w^\phi$ in $S_0$ and fully $\F$-centralized. 

Any element of $S_0 - P_1Z$ interchanges the two classes of four-subgroups of
$P_1$. Hence if $b_1 \in S_0 - P_1Z$ is of order $4$, then $b_1^2 \in Z$ and
$b_1$ induces an involutory automorphism of $P_1$ interchanging the two classes
of four-subgroups of $P_1$.  So $C_{P_1}(b_1) = \gen{z}$ and $|S\colon C_S(b_1)|
\geq |P_1\colon C_{P_1}(b_1)| \geq 8$ as $|P_1| \geq 16$. 

As $w^\phi$ is fully $\F$-centralized, the preceding paragraph implies
$w^\phi \in P_1Z = \gen{x} \times P_1$, and consequently $w^\phi = x_0v$ for
some $x_0 \in \gen{x}$ and $v \in P_1$ of order $4$. Now $x^\phi = (w^\phi)^2 =
z$, contrary to Lemma~\ref{L:x^a=xz}.  
\end{proof}

Therefore by the previous three lemmas, $|F| = 4$ and $S/P_1$ is dihedral of
order $8$.  We now obtain the final contradiction, completing the proof of
Theorem~\ref{T:2rank3}.
\begin{lemma}
$S/P_1$ is not dihedral.
\end{lemma}
\begin{proof}
Suppose it is. Again let $w \in F$ with $w^2 = x$. Then $F = \gen{w}$.  Since
$C_T(F) = PF$ is of index $2$ in $T$ and $[F,h] = \gen{z}$, we have $F \leq
Z_2(T)$ in the present situation. Moreover, $T/Z = F/Z \times P_1Z/Z$ with the
second factor dihedral of order at least $8$, and so $Z_2(T) = F \times V$
where $V$ is cyclic of order $4$ in $P_1$. Now $b$ inverts $P_1w$ by
assumption; hence $ww^b \in Z_2(T) \cap P_1 = V$. Since $Z_2(T) = F \times V$
is abelian and $w$, $w^b \in Z_2(T)$, we have $[w, w^b] = 1$. Thus, on the one
hand, $ww^b \in C_V(b) = \gen{z}$ because $b^2 = 1$. But as $Z \norm S$,
Lemma~\ref{L:x^a=xz} shows $x^b = xz$. So on the other hand, $(ww^b)^2 =
w^2(w^b)^2 = w^2(w^2)^b =  x(xz) = z$. These two facts are incompatible, and
the proof is complete.  
\end{proof}

\section{The $2$-rank $4$ case: $|Q| = 2$}\label{S:2rank4Q=2}

For a fusion system $\F$ on $S$ satisfying Hypothesis~\ref{H:main}, $T$ has
$2$-rank $3$ or $4$ by Lemma~\ref{L:2rank3or4}. By Theorem~\ref{T:2rank3}, $T$
is of $2$-rank $4$.  The current section will be devoted to the proof of the
following reduction. 

\begin{theorem}\label{T:Q>2}
Assume that $\F$ satisfies Hypothesis~\ref{H:main} with $T$ of $2$-rank $4$.
Then $|Q|>2$. 
\end{theorem}

The notation follows that begun in Section~\ref{S:centralizer}, in particular
that of Hypothesis~\ref{H:main}.  For instance, $\K$ is the unique component of
the involution centralizer $\C = C_\F(x)$, and is a fusion system of a finite
group $K$ isomorphic with $L_2(q)$ for suitable $q \equiv \pm 1 \pmod{8}$ (as
chosen once and for all after Proposition~\ref{P:dihtame}).  The Sylow subgroup
of $\K$ is denoted by $P$, a dihedral group of order $2^k$ ($k = \nu_2(q^2 - 1)
- 1 \geq 3$).  Consistent with Sections~\ref{S:2central} and \ref{S:2rank3}, we
also set $Z(P) = \gen{z}$. Denote by $C$ the cyclic maximal subgroup of $P$.
By Hypothesis~\ref{H:main}, the Thompson subgroup $J(T) = J(S)$ and so this
common subgroup is denoted simply by $J$. 

By Lemma~\ref{L:2rank3or4},
\begin{align}\
\text{there exists an involutory $f$-element $f \in C_T(P)$}.
\end{align}
We fix such an involution $f$. 

Before beginning the proof of Theorem~\ref{T:Q>2}, we collect some facts seen
before, and which hold throughout $2$-rank $4$ case. In particular,
\begin{align}\label{E:2rank4initial(c)}
J \leq R\gen{f} = Q\gen{f} \times P
\end{align}
from Lemma~\ref{L:2rank3or4}, and 
\begin{align}\label{E:2rank4initial(b)}
C_T(P) = Q\gen{f} \times \gen{z}
\end{align}
from Lemma~\ref{P:TK/Q} and the structure of $\Aut(K)$ in
Lemma~\ref{L:l2qomnibus}. Finally,
\begin{align}\label{E:T<S}
T < S
\end{align}
by Proposition~\ref{P:2centralfinal}.

Assume for the remainder of this section that $Q = \gen{x}$ is of order $2$, as
we prove Theorem~\ref{T:Q>2} by way of contradiction in a series of lemmas.
Thus, $J = \gen{x,f} \times P$ by \eqref{E:2rank4initial(c)} and $Z(J)=
\gen{x,f,z}$ is elementary abelian of order $8$.

\begin{lemma}\label{L:zweaklyclosedinZJ}
$z^{\F} \cap Z(J) = \gen{z}$. 
\end{lemma}
\begin{proof}
The Thompson subgroup $J = \gen{x,f} \times P$ is weakly $\F$-closed.  By
Lemma~\ref{L:BurnsideLemma} fusion in $Z(J)$ is controlled in $\Aut_\F(J)$.
But $\mho^1(J) \cap Z(J) = \gen{z}$ is characteristic in $J$, so the statement
follows.
\end{proof}

\begin{lemma}\label{L:zweaklyclosedinPinJ}
Let $y \in Z(J)$. Then each involution of $yP$ is $\C$-conjugate to $y$ or to
$yz$.  In particular, $z^\F \cap J = z^\F \cap P$. 
\end{lemma}
\begin{proof}
Since $\K$ has one class of involutions and $x \in C_T(\K)$ the lemma holds for
$y = x$, $z$, and $xz$. So we may assume that $y$ is an $f$-element on $\K$,
that is, $y$ centralizes $P$ but $y \nin \gen{x,z}$. Let $t$ be an involution
of $P$ so that $yt$ is also an involution. If $t = z$ then the statement is
obvious, so assume $t$ is a noncentral involution of $P$. Set $U = \gen{t,z}$,
and let $\phi \in \Aut_\K(U)$ of order $3$ such that $t^\phi = z$. Then $\phi$
extends to $\tilde{\phi} \in \C$ on $U\gen{y}$ and centralizes either $y$ or
$yz$ by Lemma~\ref{L:felements}. In the former case, $(yt)^{\tilde{\phi}} =
yz$, and in the latter, $(yt)^{\tilde{\phi}^2} = (yztz)^{\tilde{\phi}^2} = y$. This
completes the proof of the first statement. The second statement now follows
from Lemma~\ref{L:zweaklyclosedinZJ}.
\end{proof}

\begin{lemma}\label{L:PnormS}
The following hold.
\begin{enumerate}[label=\textup{(\alph{*})}]
\item $P$ is normal in $S$, 
\item there exists a fully $\F$-centralized four subgroup of $P$,
\item $[S,S] \leq C_S(C)$, and
\item no element of $S$ squares into $J-Z(J)C$.
\end{enumerate}
\end{lemma}
\begin{proof}
Let $s \in S$. Then $z^s \in Z(J)$, and so $z^s = z$ by
Lemma~\ref{L:zweaklyclosedinZJ}. But all involutions of $P^s$ are
$\K^s$-conjugate by Lemma~\ref{L:l2qomnibus}(f). Hence $P^s = \Omega_1(P^s) =
\gen{z^{\K^s}} \leq P$ by Lemma~\ref{L:zweaklyclosedinPinJ} and (a) holds.

Let $U$ be a four subgroup of $P$, and let $\psi \in \Hom_\F(U, S)$ such that
$U^\psi$ is fully $\F$-centralized. By the extension axiom, $\psi$ extends to
$C_S(U)$, which contains an elementary abelian subgroup of maximal rank. Thus,
$U^\psi \leq J$ and the nonidentity elements of $U^\psi$ consist of
$\F$-conjugates of $z$. It follows that $U^\psi \leq P$ by
Lemma~\ref{L:zweaklyclosedinPinJ}, proving (b).

Part (c) is Lemma~\ref{L:autD}(b). For (d), suppose that $s \in S$ with $s^2
\in J-Z(J)C$. Then $s^2$ must lie in $C_S(P)C$ by Lemma~\ref{L:autD}(c).  Thus,
$s^2 \in C_{J}(P)C = Z(J)C$. 
\end{proof}

\begin{lemma}\label{L:J=PCS(P)}
$J = PC_S(P)$.
\end{lemma}
\begin{proof}
Suppose the lemma is false and choose $a \in C_S(P)-J$ with $a^2 \in J$.  Let
$Z_a = C_{Z(J)}(a)$, which contains $z$ and is of order $4$.  Then $a^2 \in
Z(J) \cap C_S(a) = Z_a$ since $a$ centralizes $P$. By
\eqref{E:2rank4initial(b)} and \eqref{E:2rank4initial(c)}, $J = PC_T(P)$, and
so $a$ does not centralize $x$. 

Fix a fully $\F$-centralized four subgroup $U$ of $P$ guaranteed by
Lemma~\ref{L:PnormS}(b), a $\K$-automorphism $\phi$ of $U$ of order $3$, and an
extension $\tilde{\phi} \in \Hom_\F(C_S(U),S)$ (by the extension axiom).
Observe that 
\begin{align}
\text{$Z(J)a$ contains no involution}, 
\label{E:Z(J)a^2neq1}
\end{align}
because each element of $Z(J)a$ lies outside $J$ and centralizes the subgroup
$Z_aU$, which is of $2$-rank $3$. Also, $\tilde{\phi}$ is defined on $a$; it
follows that 
\begin{align}
a^2 \neq z,
\label{E:a^2neqz}
\end{align} 
since otherwise $a^{\tilde{\phi}}$ is an element of $S$ squaring to a
noncentral involution of $P$, contrary to Lemma~\ref{L:PnormS}(d).

If $[x, a] = z$, then $[x^{\tilde{\phi}}, a^{\tilde{\phi}}] = z^\phi$ is a
noncentral involution of $P$, contradicting Lemma~\ref{L:PnormS}(c). Finally,
we consider the case in which $[x,a] = y \neq z$. Here, $a^2 \in Z_a =
\gen{y,z}$. If $a^2 = y$, then $x$ inverts $a$, and so $(xa)^2 = 1$ contrary to
\eqref{E:Z(J)a^2neq1}. Hence $a^2 = yz$ by \eqref{E:a^2neqz}.  In this case,
$(xa)^2 = xa^2x[x,a] = xyzxy = z$, so we may replace $a$ by $xa$ to obtain $a^2
= z$, again contradicting \eqref{E:a^2neqz}, and completing the proof.
\end{proof}

Let $\Omega$ be the two element set consisting of the $P$-classes of
four-subgroups of $P$. By Lemma~\ref{L:PnormS}(a), $S$ acts on $\Omega$. Let
$N$ be kernel of this action. Then $J \leq N$. By the previous two lemmas $S/J$
embeds into $\Out(P)$. Thus by Lemma~\ref{L:autD}(a), $S/J$ has a cyclic
subgroup $B$ with index $1$ or $2$ and with $B = N/J$ cyclic of order dividing
$2^{k-3}$. Thus, $N$ is of index $1$ or $2$ in $S$. 

\begin{lemma}\label{L:modular}
We have $J < N$. In particular, $|P| \geq 16$.
\end{lemma}
\begin{proof}
Suppose to the contrary that $J=N$. Then $N \leq T$ and $T < S$ by
\eqref{E:T<S}. Since $|S:N| \leq 2$, it follows that $N = J = T$ and $|S:T|=2$. 

Fix $a \in S-T$. As $a$ acts on $Z(J) = Z(T)$ and does not centralize $x$,
$Z(S) = \gen{y,z}$ for some $y \in Z(J)-\gen{x,z}$. Since $a$ normalizes $P$
and acts nontrivially on $\Omega$, we have $[P,a] = C$. Thus, we have two
possibilities for the commutator subgroup of $S$. Either $[x,a] = z$ and $[S,S]
= C$, or else $[S,S] = Z(S)C$. 

Assume first that $[S,S] = C$. We have that $y \in Z(S)-\gen{z}$ is not
$\F$-conjugate to $z$ by Lemma~\ref{L:zweaklyclosedinZJ}. But then
Theorem~\ref{T:TT}, applied with $C$ playing the role of $T$ there, forces $z
\in y^\F$ anyway, a contradiction.

Now suppose $[S,S] = Z(S)C$, so that $[S,S] \cap Z(J) = Z(S)$ is of order $4$.
We will show in this case that $x \nin \foc(\F)$. We claim
\begin{eqnarray}
\label{E:nonmodularcent}
\mbox{ every fully } \F\mbox{-centralized conjugate of } x \mbox{ lies in } Z(J).
\end{eqnarray}
Let $s$ be a fully $\F$-centralized conjugate of $x$. Then $s$ lies in a
elementary abelian subgroup of rank $4$ by the extension axiom, so $s \in J =
Z(J)P$. If $s \in Z(J)(P-C)$, then $s$ has at least four conjugates under
$P\gen{a}$ because $Z(J) \cap P = \gen{z}$ and $P\gen{a}$ is transitive on the
involutions in $P-C$. So $|C_S(x)|>|C_S(s)|$, and
(\ref{E:nonmodularcent}) holds.

Thus, \eqref{E:nonmodularcent} implies $Z(S)Cx$ is the unique nonidentity
element of $S/Z(S)C$ containing a fully $\F$-centralized $\F$-conjugate of $x$.
This allows us to apply Theorem~\ref{T:TT}, with $Z(S)C = [S,S]$ in
the role of $T$, to obtain an $\F$-conjugate of $x$ in $\Omega_1([S,S]) =
Z(S)$, contradicting the assumption that $x$ is fully $\F$-centralized. We
conclude that $J < N$, and the first statement of the lemma holds.

For the last statement, suppose $|P| = 8$. Every element inducing an outer
automorphism on $P$ interchanges the two classes of four-subgroups of $P$.
Thus $N$ induces inner automorphisms on $P$, i.e. $N = PC_S(P) = J$,
contrary to $J < N$. Therefore, $|P| \geq 16$. 
\end{proof}

By an earlier remark and Lemma~\ref{L:modular},  $N/J$ is nontrivial cyclic.
Choose $w \in N$ mapping to a generator of $N/J$. By the definition of $N$, we
may adjust $w$ by an element of $P$ and assume that $w$ centralizes a four
subgroup $U = \gen{e, z}$ of $P$. Replacing $w$ by $ew$ if necessary, we may
assume also that
\begin{eqnarray}
w \mbox{ centralizes } C/\mho^2(C).
\label{E:wmodular}
\end{eqnarray}

Let $f_1 \in \gen{w}$ such that $f_1 \nin J$ but $f_1^2 \in J$. Since $J =
PC_S(P)$ by Lemma~\ref{L:J=PCS(P)}, we have $f_1^2 \in C_S(C)$ by
Lemma~\ref{L:autD}(d) applied with $D = P$ there. Then $f_1$ takes a generator
$c$ of $C$ to $cz$ by choice of $f_1 \in \gen{w}$ and \eqref{E:wmodular}, and
hence $f_1$ centralizes $\mho^1(C)$. As $|P| \geq 16$ from Lemma~\ref{L:modular}, it
follows that
\begin{eqnarray}
C_P(f_1) \mbox{ is the nonabelian dihedral subgroup } P_1 := \gen{U, \mho^1(C)}
\mbox{ of } P.
\label{E:C_P(f_1)dihedral}
\end{eqnarray}
So $f_1^2 \in C_S(P_1) \cap J = Z(J)$, and $f_1$ is of order at most $4$.  But
$f_1$ is not an involution, otherwise $\gen{f_1, C_{Z(J)}(f_1), U}$ is an
elementary $16$ outside $J$, so $f_1$ is of order $4$. In fact, it is shown
below that there are no involutions in $Jf_1$. For this, we will need that
\begin{eqnarray}
f_1 \mbox{ does not square to } z.
\label{E:f1doesnotsquaretoz}
\end{eqnarray}
Assume to the contrary that $f_1^2 = z$. Then as $[\mho^1(C), f_1] = 1$ and $|P| \geq
16$ by Lemma~\ref{L:modular}, there exists an element $v \in C_P(f_1)$ with
$(vf_1)^2 = 1$.  But then $C_P(vf_1)$ contains the four subgroup $\gen{ce, z}$
of $P$, and so $vf_1 \in J$, yielding the same contradiction as before and thus
confirming (\ref{E:f1doesnotsquaretoz}).

\begin{lemma}\label{E:nomodularinvolution}
There is no involution in $Jf_1$.
\end{lemma}
\begin{proof}
By \eqref{E:f1doesnotsquaretoz}, we have $f_1^2 \in Z(J)-\gen{z}$.
If $[Z(J), f_1] \leq \gen{z}$, then $J\gen{f_1}/P \cong C_2 \times C_4$
with $Pf_1$ of order $4$. In this case, every element of order $2$ in
$J\gen{f_1}$ lies in $J$ as claimed.  Hence we may assume that $[Z(J),
f_1] = \gen{y} \neq \gen{z}$. Fix an element $y_1 \in Z(J) - \gen{y,z}$. Then
$y_1$ inverts $f_1$ and $D_1 :=  \gen{y_1, f_1}$ is dihedral of order $8$ with
center $\gen{y}$.  Hence $Z(J)P_1\gen{f_1} = D_1 \times P_1$ is of $2$-rank $4$
and equal to its Thompson subgroup. This forces $D_1 \leq J$. But then
$D_1 \leq C_J(P_1) = Z(J)$, a contradiction.
\end{proof}

In addition, we let $h_1 \in S-N$ be an element such that $h_1^2 \in J$ or set
$h_1 = 1$ if such an element does not exist.  Note that if $h_1 = 1$, then
$S/J$ is cyclic by the structure of $\Out(P)$ (Lemma~\ref{L:autD}(a)).  In any
case, $S/J\gen{h_1}$ is cyclic.  

Assume that $h_1 \neq 1$. Then 
\begin{eqnarray}
\mbox{both } h_1 \mbox{ and } h_1f_1 \mbox{ square into } J.
\label{E:h_1andh_1f_1squares}
\end{eqnarray}
Let $s \in Jh_1 \cup Jh_1f_1$. Then for any $e_1 \in P-C$, we have $e_1^s =
e_1c$ for some generator $c$ of $C$ because $s \nin N$.  Therefore,
\begin{eqnarray}
    [P,s] = C \mbox{ and } C_{J}(s) \leq Z(J)C \mbox{ is abelian.}
    \label{E:nonmodularcents}
\end{eqnarray}
Furthermore, as $e_1^{s^2}=e_1cc^s$, we have
\begin{eqnarray}
\mbox{if } s^2 \in C_S(P) \mbox{ then } s \mbox{ inverts } C.
\label{E:sinvertsC}
\end{eqnarray}

With this setup, the next two lemmas contradict each other and complete the
proof of Theorem~\ref{T:Q>2}.
\begin{lemma}\label{L:involutionsbothcosets}
$h_1 \neq 1$ and both cosets $Jh_1$ and $Jh_1f_1$ contain involutions.
\end{lemma}
\begin{proof}
Suppose either that $h_1 = 1$ or that there are no involutions in $Jh_1f_1$.
The argument is the same in case $h_1 \neq 1$ and $Jh_1$ contains no
involutions. (Alternatively, swap the roles of $h_1$ and $h_1f_1$ in this extra
case.) By Lemma~\ref{E:nomodularinvolution} and assumption, $\Omega_1(S) \leq
J\gen{h_1}$.  Also $S/J\gen{h_1}$ is cyclic, and $f_1$ is of least order
outside $J\gen{h_1}$. By Theorem~\ref{T:TT}, there exists a morphism
$\phi \in \F$ such that $f_1^\phi \in J\gen{h_1}$ is fully $\F$-centralized,
and $C_S(f_1)^\phi \leq C_S(f_1^\phi)$.

Now if $h_1 \neq 1$, then $f_1^\phi$ cannot lie in the coset $Jh_1$.  This is
because $\Omega_1(C_P(f_1))$ is nonabelian dihedral by
\eqref{E:C_P(f_1)dihedral}, whereas $\Omega_1(C_S(s))$ is abelian for every $s
\in Jh_1$ by (\ref{E:nonmodularcents}).  So $f_1^\phi \in J$ whether or not
$h_1=1$. Since $f_1$ is of order $4$ and $\Omega_1(\mho^1(J)) = \gen{z}$, we
have that $(f_1^2)^\phi = z$. But $z$ is weakly $\F$-closed in $Z(J)$ by
Lemma~\ref{L:zweaklyclosedinZJ} and so $f_1^2 = z$, contrary to
(\ref{E:f1doesnotsquaretoz}).  
\end{proof}

\begin{lemma}
$Jh_1f_1$ contains no involution. 
\end{lemma}
\begin{proof}
We have $h_1 \neq 1 = h_1^2$ by Lemma~\ref{L:involutionsbothcosets}, and then
$h_1$ inverts $C$ by (\ref{E:sinvertsC}). So 
\begin{eqnarray}
h_1f_1 \mbox{ sends a generator } c \mbox{ of } C \mbox{ to } c^{-1}z.
\label{E:actionofh_1f_1}
\end{eqnarray}
In particular, $C_C(h_1f_1) = \gen{z}$ and so $C_{Z(J)C}(h_1f_1) \leq
Z(J)\Omega_2(C)$.  As $(h_1f_1)^2 \in J$ from \eqref{E:h_1andh_1f_1squares}, we
have $(h_1f_1)^2 \in C_J(h_1f_1) = C_{Z(J)C}(h_1f_1) \leq Z(J)\Omega_2(C)$ with
the equality by (\ref{E:nonmodularcents}).  But $(h_1f_1)^2$ does not lie in
$C_S(P) = Z(J) = \Omega_1(Z(J)\Omega_2(C))$ by (\ref{E:sinvertsC}), since
$h_1f_1$ does not invert $C$.  Consequently, 
\begin{eqnarray}
\label{E:h1f1order8}
h_1f_1 \text{ is of order } 8.
\end{eqnarray}
Set $M = J\gen{h_1f_1}$ and $\bar{M} = M/Z(J)$. Then $\bar{M}$ contains the
dihedral group $\bar{P}$ as a maximal subgroup, which is nonabelian as $|P|
\geq 16$. Furthermore $\bar{M}$ is of maximal class by
(\ref{E:nonmodularcents}) and (\ref{E:actionofh_1f_1}), and so $\bar{M}$ is
dihedral or semidihedral. As $\bar{h_1f_1}$ is of order $4$ squaring into the
center of $\bar{M}$ by \eqref{E:h1f1order8}, and as a dihedral group of order
at least $16$ has no such element outside a maximal dihedral subgroup, we know
$\bar{M}$ is semidihedral.  But then $\bar{M}$ contains no involutions at all
outside its dihedral maximal subgroup $\bar{P}$. It follows that $M =
J\gen{h_1f_1}$ contains no involutions outside $J$, which is what was to be
shown.
\end{proof}

\section{The $2$-rank $4$ case: $|Q|>2$}\label{S:2rank4Q>2}

For this final section, we continue to assume $\F$ is a saturated fusion system
on the $2$-group $S$ satisfying Hypothesis~\ref{H:main}. By the main results of
the previous three sections, we are reduced to the following situation in
describing $\F$.

\begin{enumerate}
\item
$T = C_S(x)$ is a proper subgroup of $S$ (Proposition~\ref{P:2centralfinal}),
\item
$S$ is of $2$-rank $4$ (Theorem~\ref{T:2rank3}), and
\item
$Q = C_T(\K)$ is of order at least $4$ (Theorem~\ref{T:Q>2}).
\end{enumerate}

\begin{theorem}\label{T:2rank4main}
Let $\F$ be a saturated fusion system on the $2$-group $S$. Assume $\F$
satisfies Hypothesis~\ref{H:main} and, in addition, the above three items. Then
$S \cong D_{2^k} \wr C_2$, and $\F$ is the fusion system of $L_4(q_1)$ for some
$q_1 \equiv 3 \pmod{4}$ with $\nu_2(q_1+1) = k-1$.
\end{theorem}

Adopt the notation of Hypothesis~\ref{H:main} and the set up at the beginning of
Section~\ref{S:2rank4Q=2}.  By Lemma~\ref{L:2rank3or4},
\[
\text{there exists an involutory $f$-element $f \in C_T(P)$}.
\]
We continue to fix such an involution $f$. As $T$ a proper subgroup of $S$, we also
\[
\mbox{fix } a \in N_S(T)-T \mbox{ with } a^2 \in T.
\]

As usual, we prove Theorem~\ref{T:2rank4main} in a sequence of lemmas.  It will
emerge quickly (after Lemma~\ref{L:Qfdihedral}) that $J = R\gen{f}$ is the
product of two dihedral groups $Q\gen{f}$ and $P$ of the same order, $T$ has
index $2$ in $S$, and $T/R$ is of exponent $2$.  Since $T/R$ embeds into
$\Out(K)$ (Proposition~\ref{P:TK/Q}), this means that either $T=J=R\gen{f}$, or
$T = R\gen{h,f}$ for some $1 \neq h \in T$ such that $h^2 \in Q$ and
$R\gen{h}/Q$ is dihedral of order $2|P|$. Thus $R\gen{h}\K/Q$ is uniquely
determined as the fusion system of $PGL_2(q)$ (see the description of $\Out(K)$
in Lemma~\ref{L:l2qomnibus}). In anticipation of this we let
\begin{align}\label{E:fixh}
h &\in T-R \mbox{ such that } h^2 \in Q \mbox{ and } R\gen{h}/Q \mbox{ is
dihedral, or}  \\ \notag
h &= 1 \mbox{ if such an element does not exist. }
\end{align}
Note in the case $h \neq 1$,
\begin{align}
[Qh, Qf] = Qz
\label{E:[h,f]}
\end{align}
by Lemma~\ref{L:l2qomnibus}(h).

Much of the $2$-group and transfer analysis will be dedicated to analyzing
whether or not $h$ exists, and if it does, whether $Q\gen{h}$ splits over $Q$.
Together with the target $\F = \F_S(L_4(q_1))$ appearing within the case $T = J$,
Table~\ref{table} lists the fusion systems (of finite groups) which nearly
satisfy the conditions of Theorem~\ref{T:2rank4main}, and why they are
eventually ruled out.

\begin{table}
\centering
\begin{tabular}[]{c|c|c|c|c}
Scenario & $T$ & $S$ & Group & Contradiction \\
\hline
$h \neq 1$ and $Q\gen{h}$ is dihedral & $J\gen{h}$ & $Q_{2^{k+1}} \wr_* C_2$ &
$PSp_4(q)$ & $C_T(\K) \cong D_{2^k}$ \\ 
$h \neq 1$ and $Q\gen{h}$ is cyclic & $J\gen{h}$ & $SD_{2^{k+1}} \wr_* C_2$ &
$PGL_4(q_1)$ & $O^2(\F) < \F$ \\
$h = 1$ & $J$ & $D_8 \wr C_2$ & $A_{10}$ & $C_T(\K) \cong C_2 \times C_2$ \\
$h = 1$ & $J$ & $D_{2^k} \wr C_2$ & $L_4(q_1)$ & \\
\end{tabular}
\caption{Outcomes based on the structure of $Q\gen{h}$}
\label{table}
\end{table}

\noindent
The $2$-group and transfer analysis is carried out through Lemma~\ref{L:h^2=Q},
where it is shown that $S$ is isomorphic to a Sylow $2$-subgroup of $PSp_4(q)$
or $PGL_4(q_1)$ when $h \neq 1$.  Then we compute the centralizer of a central
involution via an argument modeled on that of
\cite[Lemmas~3.15, 3.16]{GorensteinHarada1973} and apply transfer
(Theorem~\ref{T:TT}) to rule out the $PGL_4(q_1)$-case.  Analyzing the resulting
fusion information allows us to conclude that $h=1$, and $S$ is then
isomorphic to $D_{2^{k}} \wr C_2$.  Lastly, we appeal to a result of Oliver
\cite{OliverSectRank4} to identify $\F$ as the fusion system of $L_4(q_1)$.

We begin by pinning down the structure of $J$ in the next few lemmas. 
\begin{lemma}\label{L:Qfdihsemidih}
The following hold.
\begin{itemize}
\item[(a)] $Q\gen{f}$ is dihedral or semidihedral,
\item[(b)] $J = \Omega_1(Q\gen{f}) \times P$,
\item[(c)] no element of $Rf$ is a square in $T$, and
\item[(d)] $T = R\gen{h,f}$.
\end{itemize}
\end{lemma}
\begin{proof}
We claim that $Z(Q\gen{f}) = \gen{x}$. Suppose this is not the case. Then
either $f$ centralizes $Q$ or $|Q| \geq 8$ and $f$ acts on $Q$ by sending a
generator $d$ of $Q$ to $dx$.  In either case we have that $J = \gen{x, f}
\times P$ by \eqref{E:2rank4initial(c)}.  Then $D = C_T(J)$ is normal in
$S$, and $D$ is equal to $Q \times \gen{f, z}$ if $f$ centralizes $Q$ and to
$\mho^1(Q) \times \gen{f, z}$ otherwise.  In any case, $x$ is the only
involution which is a square in $D$, so $x \in Z(S)$. This contradicts $T < S$.
Thus $Q\gen{f}$ is of maximal class and $f$ is an involution outside $Q$, so
(a) holds. Now (b) follows by \eqref{E:2rank4initial(c)}. 

Suppose some element $f_1 \in T$ squares into $Rf$. Write $\tilde{T} = T/P$ and
denote images modulo $P$ similarly. Then $D_1 := \tilde{Q}\gen{\tilde{f_1}}$
contains $D := \tilde{Q}\gen{\tilde{f}}$ as a normal subgroup. Moreover as
$Q\gen{f} \cap P = 1$, $D \cong Q\gen{f}$. If $Q\gen{f}$ is dihedral in (a),
then Lemma~\ref{L:autD}(c), with $D_1$ in the role of $S$ there, gives a
contradiction.  If $Q\gen{f}$ is semidihedral, then considering the images of
$D_1$ and $D$ modulo $\gen{\tilde{x}}$, we obtain the same contradiction to
Lemma~\ref{L:autD}(c).  Therefore, (c) holds.  By the structure of $\Out(K)$
and (c), $T/R$ is a four group covered by $\gen{h,f}$, yielding (d).
\end{proof}

\begin{lemma}\label{L:Sclassofx}
$x^S = \set{x, z}$. Consequently, $|S:T|=2$.
\end{lemma}
\begin{proof}
By Lemma~\ref{L:Qfdihsemidih}, $S$ normalizes $Z(J) = \gen{x, z}$ and $J =
J(Q\gen{f}) \times P = \Omega_1(Q\gen{f}) \times P$ is the product of two
nonabelian dihedral groups. By Proposition~\ref{P:krullschmidt}, $S$ must permute
the commutator subgroups of these factors. Therefore, $x^S \cin \set{x, z}$ and
as $T = C_S(x)$ is a proper subgroup of $S$, we have equality and $|S:T|=2$. 
\end{proof}

Recall $a \in S-T$ has been fixed, squaring into $T$. By the previous lemma, $T
\norm S = T\gen{a}$ and $a$ swaps $x$ and $z$. So $Z(S) = \gen{xz}$, and
\begin{align}
x \mbox{ is not } \F\mbox{-conjugate to } xz
\label{E:xnsimxz}
\end{align}
because $x \nin Z(S)$ is fully $\F$-centralized.

\begin{lemma}\label{L:PcentralizesP^a}
$[P^a, P] = P^a \cap P = 1$.
\end{lemma}
\begin{proof}
Note first that both $P$ and $P^a$ are normal in $T$, So $[P, P^a] \leq P \cap
P^a$.  Suppose that $Z_0 := P \cap P^a \neq 1$. Then $Z_0$ is nontrivial normal
in $P^a$, and so $\gen{x} = Z(P^a) \leq Z_0 \leq P$, a contradiction.  
\end{proof}

\begin{lemma}\label{L:Qfdihedral}
$Q\gen{f}$ is dihedral of the same order as $P$. 
\end{lemma}
\begin{proof}
Recall from Lemma~\ref{L:Qfdihsemidih} that $\Omega_1(Q\gen{f})$ is nonabelian
dihedral and $J = \Omega_1(Q\gen{f}) \times P$. Let $C_1$ and $C_2$ be the
cyclic maximal subgroups of $\Omega_1(Q\gen{f})$ and $P$, respectively. Because
$x^a = z$, we have $C_1^a$ is a cyclic subgroup of $J$ with $z$ as its unique
involution, and so $|C_1| \leq |C_2|$ by the structure of $J$.  We conclude
similarly that $|C_2| \leq |C_1|$ by considering $C_2^a$, and hence $C_1$ and
$C_2$ are of the same order. Therefore either $Q\gen{f}$ is either dihedral of
the same order as $P$, or $Q\gen{f}$ is semidihedral with $|Q\gen{f}| = 2|P|$.

Suppose $Q\gen{f}$ is semidihedral. Then $\Omega_1(Q\gen{f}) \cong P$. Set $S_1
= (Q\gen{f})^a$ for short. Then $S_1$ centralizes $P^a$.  But $P$ also
centralizes $P^a$ by Lemma~\ref{L:PcentralizesP^a}. In fact,
\[
P^a \times \gen{z} \leq \Omega_1(C_T(P)) = \Omega_1(Q\gen{f}) \times \gen{z}
\]
so the above three subgroups are equal, as the outside two are of the same
order. Taking centralizers, we get that 
\begin{align}
S_1 \leq C_T(P^a) = C_T(\Omega_1(Q\gen{f})) \leq C_T(f)
\label{E:S1}
\end{align}
Now $f$ centralizes no element in the coset $Rh$ when $h \neq 1$ by
\eqref{E:[h,f]}. So $C_T(f) = C_{R\gen{f}}(f) = P \times \gen{f,x}$ by
Lemma~\ref{L:Qfdihsemidih}(d,a), and \eqref{E:S1} is a contradiction because $P
\times \gen{f,x}$ contains no semidihedral subgroup.
\end{proof}

In view of the previous lemma, it is now determined that
\begin{align}
J = Q\gen{f} \times P = P^a \times P
\label{E:J}
\end{align}
with $a$ interchanging $P$ and $P^a$. Since $P^a \leq C_T(P) = Q\gen{f} \times
\gen{z}$, we may replace $f$ by $fz$ and assume that
\begin{align}
\label{E:finP^a}
f \in P^a.
\end{align}
We fix notation for the maximal cyclic subgroup of $P$, calling it $C$.
Then
\begin{align}
\text{$C^a \leq Q\gen{z}$ and $Q^a \leq C\gen{x}$}
\label{E:C^a=Q<z>}
\end{align}
by the above remarks.

The next lemma shows that $a$ may be chosen to be an involution. Part (a) of it
will later be shown in Lemma~\ref{L:h=1} to rule out the $PSp_4(q)$-case
mentioned above and determine that in fact $h=1$.

\begin{lemma}\label{L:a^2=1}
The following hold.
\begin{enumerate}[label=\textup{(\alph{*})}]
\item If $C = \gen{c}$, then $ff^a$ is not $\F$-conjugate to $f(cf^a)$. 
\item There exists an involution in $S-T$. 
\item Each involution in $J$ is $\F$-conjugate to an element of $\gen{x,z}$
\end{enumerate}
\end{lemma}
\begin{proof}
We will show that $ff^a$ is not conjugate to $f(cf^a)$ from the fact that one
of them is $\F$-conjugate to $x$ and the other to $xz$. Recall from
\eqref{E:finP^a} we have chosen $f \in P^a$, so $f^a \in P$.  Thus, $U_0 =
\gen{f^a,z}$ and $U_1 = \gen{cf^a, z}$ are four-subgroups of $P$ which are not
$P$-conjugate.  Since $f$ is an $f$-element on $\K$
(Definition~\ref{D:felement}), $C_\C(f)$ contains $\Aut_\K(U_j)$ for some $j$,
and $C_\C(fz)$ contains $\Aut_\K(U_{1-j})$ by Lemma~\ref{L:felements}(b). Thus,
there is an element $\phi \in \Aut_\C(\gen{f}U_j)$ of order $3$ with $(f\cdot
c^jf^a)^\phi = f^\phi(c^jf^a)^\phi = fz$. On the other hand, there is a similar
element $\psi \in \Aut_\C(\gen{f}U_{1-j})$ with $(fz\cdot zc^{1-j}f^a)^\psi =
fz \cdot z = f$.  Now as $f$ is $\K^a$-conjugate to $x$ in $C_\F(z)$, $fz$ is
$C_\F(z)$-conjugate to $xz$. So \eqref{E:xnsimxz} shows $ff^a$ and $f(cf^a)$
are not $\F$-conjugate. establishing (a).

For (b), suppose that $a^2 \in J$. It will be shown first that (b) holds in
this situation.  Write $a^2 = t\inv{s}$ with $t \in P^a$ and $s \in P$.  Let
$a_0 = as$. Then $P^{a_0} = P^a$, and $a_0^2 = a^2s^as = ts^a$ as $[P, P^a] =
1$. So $a_0^2 \in P^a$ and centralizes $a_0$. Therefore $a_0^2 = 1$, as
claimed.

So it remains to prove that $a^2 \in J$. Note $T = J\gen{h}$ by \eqref{E:J} and
Lemma~\ref{L:Qfdihsemidih}(d).  Thus, if $a$ does not square into $J$, then $h
\neq 1$ and $a$ squares into the coset $Jh$, so that $S/J$ is cyclic of order
$4$. In this case, let $\mathscr{J}$ denote the set of $J$-classes of
``noncentral diagonal'' involutions of $J$, that is, those involutions in $J$
outside the set $I = P\gen{x} \cup P^a\gen{z}$.  Thus $\mathscr{J}$ has
cardinality $4$, and for any generator $c$ of $C$, the set $\set{ff^a,
(c^af)f^a, f(cf^a), c^af(cf^a)}$ is a set of representatives for the members of
$\mathscr{J}$. Since $I$ is a normal subset of $S$ and $J$ is a normal
subgroup, $S$ acts on $\mathscr{J}$ by conjugation.  Moreover, any element in
$Jh$ swaps the two $P$-classes of noncentral involutions in $P$, and so acts
nontrivally on $\mathscr{J}$. It follows that $\gen{a}$ acts transitively as a
four-cycle on $\mathscr{J}$, and hence all involutions in $J-I$ are
$S$-conjugate. This contradicts part (a) and completes the proof of (b).

For (c), keep the notation of the previous paragraph, and let $t$ be a
noncentral involution of $J$. If $t \in I$, then $t$ is $\C$ or $\C^a$
conjugate into $\gen{x,z}$. On the other hand, if $t$ is in $J-I$, then $t$ is
$J$ conjugate to some member of $\mathscr{J}$. The proof of part (a) shows that
$ff^a$ and $f(cf^a) \overset{a}{\sim} c^aff^a$ are to $f$ or $fz$ and hence
into $\gen{x,z}$.  For $t = (c^af)(cf^a)$, by \eqref{E:C^a=Q<z>} one of $c^a$
or $c^az$ lies in $Q$, which is normal in $\C$. Hence the same argument as in
(a) gives that $c^az \cdot f(zcf^a)$ is $\C$-conjugate to $c^af$ or $c^afz$,
both of which lie in $I$.
\end{proof}

From now on, we assume $a^2 = 1$.  We narrow down the structure of $T$ to two
possibilities in the next lemma, depending on whether $Q\gen{h}$ splits over
$Q$ or not, as described in the introduction to this section.
\begin{lemma}\label{L:htwocases}
Suppose $h \neq 1$. Then one of the following holds.
\begin{itemize}
\item[(a)] $h^2 = 1$ and $Q\gen{h}$ is dihedral, or
\item[(b)] $Q = \gen{h^2}$.
\end{itemize}
\end{lemma}
\begin{proof}
Recall that $Q \norm T$ and $h^2 \in Q$ by the choice of $h$.  By \eqref{E:J}
we have that $J/P^a$ is isomorphic to $P$.  From Lemma~\ref{L:l2qomnibus}(h)
and \eqref{E:C^a=Q<z>}, $T/P^a\gen{z}$ nonabelian dihedral (of order $|P|$)
with $P^aZ_2(C)/P^a = Z(T/P^a\gen{z})$. Hence $T/P^a$ is of maximal class.

Suppose that $Q \neq \gen{h^2}$. Then $h^2 \in \mho^1(Q) = \mho^1(C^a) \leq
C^a$, with the equality by \eqref{E:C^a=Q<z>}. Hence $T/P^a$ contains $J/P^a$
and an involution $P^ah$ outside $J/P^a$ inverting $P^aC/P^a$; i.e. $T/P^a$ has
two dihedral maximal subgroups.  By inspection of the maximal subgroups of
$2$-groups of maximal class, it follows that $T/P^a$ is dihedral. 

Since $T = R\gen{h,f}$ and $J = R\gen{f}$, the coset $P^ah^a$ lies outside the
$J/P^a$. Thus either $P^ah^a$ is an involution in $T/P^a - J/P^a$, or $P^ah^a$
squares to a generator of the cyclic maximal subgroup $P^aC/P^a$ of $J/P^a$. 

Suppose that $P^ah^a$ squares to a generator of $P^aC/P^a$. Then $h^a$ and
hence $h$ has order at least $2|C|$. But $h^2 \in Q$ and $|Q| = |C|$, so we
must have $Q = \gen{h^2}$, contradicting our assumption on $Q$. 

Suppose that $P^ah^a$ is an involution in $T/P^a - J/P^a$. Then $P^ah^a$
inverts $P^aC/P^a \cong C$, and so $h^a$ inverts $C$ as $C \norm T$. It follows
that $h$ inverts $C^{\inv{a}} \leq Q \times \gen{z}$. Since $h$ normalizes $Q$,
$h$ must invert $Q$. As $P^ah^a$ is an involution, we have $(h^a)^2 \in P^a$.
So $h^2 \in P$. But $h^2 \in Q$, so $h^2 \in Q \cap P = 1$ giving (a). 
\end{proof}

Set $J_0 = QC = Q \times C = C^a \times C$, a homocyclic normal subgroup of
$S$. It will be helpful for what follows to call attention to the action of $T$
on $J_0$, and describe what this means for the structure of the quotient
$S/J_0$.  Recall that $C^a \leq Q\gen{z}$ from \eqref{E:C^a=Q<z>}, and so the
action of an element of $T$ on $C^a$ is the same as on $Q$. From
Lemma~\ref{L:Qfdihedral} and the two possibilities in Lemma~\ref{L:htwocases},
each element in $T$ centralizes $C$ or inverts it, and the same holds for $C^a$
in place of $C$. Conjugation by an element in $S-T$ swaps the actions.
Moreover, $T/J_0$ is elementary abelian of order $8$ when $h \neq 1$, and $a$
induces an automorphism of $T/J_0$ fixing pointwise a four group.  For
instance, from the actions of $h$, $f$, and $f^a$ on $J_0$, and since $J \norm
S$,
\begin{align}
\mbox{
if $h$ inverts $Q$, then $\gen{h, ff^a}$ covers $C_{T/J_0}(a)$.
}
\label{E:QhdihedralS/J0}
\end{align}
and
\begin{align}
\mbox{
if $h$ centralizes $Q$, then $\gen{fh, ff^a}$ covers $C_{T/J_0}(a)$,
}
\label{E:QhcyclicS/J0}
\end{align}
Lastly, 
\begin{align}
[S,S] = J_0\gen{ff^a}.
\label{E:[S,S]}
\end{align}

Suppose $h$ is involution as in Lemma~\ref{L:htwocases}(a) from now through the
next lemma.  From \eqref{E:QhdihedralS/J0}, we have $[h,a] \in J_0 = C^aC$ and
we may arrange to have $[h,a] \in C^a$ by replacing $h$ by an appropriate
element in $Ch$. (If $[h,a] = uv$ for $u \in C^a$ and $v \in C$, set $h_0 =
v^{-1}h$. Then $[h_0,a] = h_0h_0^a = v^{-1}v^ahh^a = uv^a \in C^a$.) Hence,
\[
[h,a] = (ha)^2 \in C^a \cap C_{J_0}(ha) = 1,
\]
and $h$ still squares to the identity.  Fix this choice for $h$ now through the
next lemma. Thus, $S$ is a split extension of $J$ by the four group
$\gen{h,a}$.

We can now write down a presentation for $S$. The rest of the following lemma
is verified by direct computation, or by appeal to
\cite[Lemma~3.5]{GorensteinHarada1973}.
\begin{lemma}\label{L:h^2=1}
Suppose $h$ is an involution. Then $S$ has presentation
\begin{align*}
\gen{\,\,d,c,f,e,h,a \mid & \,\, d^{2^{k-1}} = c^{2^{k-1}} = f^2 = e^2 = a^2 = 1, \\
    & [d,c] = [f,e] = 1, \\
    & d^f = d^{-1}, \, c^{e} = c^{-1}, \, c^a = d, \, e^a = f, \\
    & h^2 = 1, \, e^h = ec, \, h^a = h \,\,}
\end{align*}
with notation consistent with that fixed. Here, $P = \gen{c,e}$, $P^a =
\gen{d,f}$, $J = P^aP$, $J_0 = \gen{d,c}$, $x = d^{2^{k-2}}$, $z =
c^{2^{k-2}}$, $Z(S) = \gen{xz}$, and $T = \gen{d,c,f,e,h}$. Furthermore, the
following hold.
\begin{enumerate}[label=\textup{(\alph{*})}]
\item $J\gen{a} \cong D_{2^{k}} \wr C_2$. 
\item $Q = \gen{dz}$.
\item $h$ inverts $J_0$ and all involutions of $Jh$ are $J$-conjugate.
\item $C_S(h) = \gen{h} \times B_0$ where $B_0 = \gen{x,a}$ is dihedral of order $8$.
\item $C_S(a) = \gen{a} \times B_a$ where $B_a = \gen{ff^ah, h}$ is dihedral of order $2^{k+1}$.
\item $C_S(ha) = \gen{ha} \times B_{ha}$ where $B_{ha} = \gen{[h,f]ff^ah,
h}$ is dihedral of order $2^{k+1}$.
\item All involutions of $Ja$ are $J$-conjugate as are all involutions of
$Jha$.
\item $C_S(ff^a) = \gen{f,x,a} \cong (C_2 \times C_2) \wr C_2$ and all
elements of $J_0ff^a$ are $S$-conjugate.
\item $D_1 := \gen{ff^ah, xff^aa}$ and $D_2 := \gen{[h,f]ff^ah, h}$ are
quaternion of order $2^{k+1}$ with $[D_1, D_2] = 1$, $D_1 \cap D_2 = \gen{xz}$,
$D_1^f = D_2$, $D := D_1D_2 = J_0\gen{ff^a,h,a} = [S,S]\gen{h,a}$, and $S =
D\gen{f} \cong Q_{2^{k+1}} \wr_* C_2$ is isomorphic with a Sylow $2$-subgroup of
$PSp_4(q)$. 
\end{enumerate}
\end{lemma}

From now through the next lemma, assume $Q = \gen{h^2}$ as in
Lemma~\ref{L:htwocases}(b). We adjust $h$ slightly as follows.  As a
consequence of \eqref{E:[S,S]}, we have that $h^a \in J_0ff^ah = CC^aff^ah$.
Since $a^2 = 1$ and $ff^a$ inverts $J_0$, we may write $h^a = c_1^ac_1ff^ah$
for some $c_1 \in C$.  Replacing $h$ by $(c_1^a)^{-1}h$, which lies in
$Q\gen{z}h$ by \eqref{E:C^a=Q<z>}, we arrange (again using that $ff^a$ inverts
$J_0$) that 
\begin{align}
h^a = ff^ah, 
\label{E:h^a}
\end{align}
and $h$ still squares to a generator of $Q$. Fix this choice of $h$ through the
next lemma.  Multiply \eqref{E:h^a} by $f^a$ on the left to obtain
\begin{align} 
[fh, a] = 1.  
\label{E:[fh,a]=1}
\end{align}
As $f \in P^a$ (see \eqref{E:finP^a}), $f^a$ is a noncentral involution in $P$.
Hence $c := [f^a,h]$ of $C$ is a generator of $C$ by
Lemma~\ref{L:l2qomnibus}(c).  Fix this choice and set $d = c^a$. Then 
\begin{align} 
d = c^a = [f,h^a] = [f,ff^ah] = [f,h].
\label{E:d=[f,h]}
\end{align}
From \eqref{E:[fh,a]=1} and the fact that $(fh)^2 \in Qz$ in \eqref{E:[h,f]},
we have $(fh)^2 = xz$, and so $xz = fhfh = h^{-2}[f,h] = h^{-2}d$ as $f$
inverts $Q$.  Hence,
\begin{align}
h^2 = dxz.
\label{E:h^2}
\end{align}
Lastly, from $(fh)^2 = xz = [x,a]$ and \eqref{E:[fh,a]=1}, we have
\[
\text{$xfha \in Jha$ is an involution}.
\]

We can now write down a presentation for $S$ in case $Q = \gen{h^2}$. The rest
of the following lemma is similarly verified by direct computation.
\begin{lemma}\label{L:h^2=Q}
Suppose $Q = \gen{h^2}$. Then $S$ has presentation
\begin{align*}
\gen{\,\,d,c,f,e,h,a \mid & \,\, d^{2^{k-1}} = c^{2^{k-1}} = f^2 = e^2 = a^2 = 1, \\
    & [d,c] = [f,e] = 1, \\
    & d^f = d^{-1}, \,\, c^{e} = c^{-1}, \,\, c^a = d, \,\, e^a = f \\
    & h^2 = dd^{2^{k-2}}c^{2^{k-2}}, \,\, e^h = ec, \,\, h^a=feh \,\,}
\end{align*}
with notation consistent with that fixed. Here $P = \gen{c,e}$, $P^a =
\gen{d,f}$, $J = PP^a$, $J_0 = \gen{d,c}$, $x = d^{2^{k-2}}$, $z =
c^{2^{k-2}}$, $Z(S) = \gen{xz}$, and $T = \gen{d,c,f,e,h}$. Furthermore, the
following hold.
\begin{enumerate}[label=\textup{(\alph{*})}]
\item $J\gen{a} \cong D_{2^k} \wr C_2$.
\item $Q = \gen{dxz} = \gen{dz}$.
\item There are no involutions in $Jh$.
\item $fh$ inverts $J_0$; all elements of $J_0fh$ square to $xz$ and are $J$-conjugate.
\item $C_S(a) = \gen{a} \times B_a$ where $B_a = \gen{f^ah, ff^a}$ is semidihedral
of order $2^{k+1}$ with $Z(B_a) = \gen{xz}$.
\item Set $b_1 = xfha \in Jha$. Then $b_1^2 = 1$ and $C_S(b_1) = \gen{b_1} \times
B_{ha}$ where $B_{ha} = \gen{d^{-1}f^ah, xa}$ is semidihedral of order
$2^{k+1}$ with $Z(B_{ha}) = \gen{xz}$.
\item All involutions of $Ja$ are $J$-conjugate as are all involutions of $Jha$.
\item $C_S(ff^a) = \gen{x,f,a} \cong (C_2 \times C_2) \wr C_2$ and all elements
of $J_0ff^a$ are $S$-conjugate.
\item $D_1 = \gen{f^ah,ff^aa}$ and $D_2 = \gen{d^{-1}f^ah, a}$ are semidihedral
of order $2^{k+1}$ with $[D_1, D_2] = 1$, $D_1 \cap D_2 = \gen{xz}$, $D_1^f =
D_2$, $D := D_1D_2 = J_0\gen{ff^a, fh, a} = [S,S]\gen{fh,a}$ and $S =
D\gen{f} \cong SD_{2^{k+1}} \wr_* C_2$ is isomorphic with a Sylow $2$-subgroup
of $PGL_4(q)$ for $q \equiv 3 \pmod{4}$ with $\nu_2(q+1) = k-1$. 
\end{enumerate}
\end{lemma}

Armed with these data, the centralizer of the central involution is computed next.
\begin{lemma}\label{L:centinvcent}
Suppose $h \neq 1$. Then $C_\F(xz)$ is realizable by a finite group $G$ having
Sylow $2$-subgroup $S$ and with the property that $G$ contains a normal
subgroup isomorphic to $SL_2(q)*SL_2(q)$ of index $2$ with $f$ interchanging
the two $SL_2(q)$ factors.  In particular, $S \cong Q_{2^{k+1}} \wr_* C_2$ and
$h$ is an involution.
\end{lemma}
\begin{proof}
Assume that $h \neq 1$. The two possibilities for $S$ in Lemmas~\ref{L:h^2=1}
and \ref{L:h^2=Q} will be treated simultaneously.  Fix $t \in J$ such that
$tha$ is an involution as follows. When in the case of Lemma~\ref{L:h^2=1}, we
take $t=1$. In the other case, we take $t = xf$ as in Lemma~\ref{L:h^2=Q}(f).
In either case $th$ commutes with $a$ and inverts $J_0$. Let $b$ be one of $a$
or $tha$.  Then $C_S(b) = \gen{b} \times B$ where $B$ is dihedral or
semidihedral of order $2^{k+1}$ with $Z(C_S(b)) \cap [C_S(b), C_S(b)] = Z(S)$
by Lemma~\ref{L:h^2=1}(e,f) and Lemma~\ref{L:h^2=Q}(e,f). Moreover, all
involutions of $Jb$ are $J$-conjugate by Lemma~\ref{L:h^2=1}(g) and
Lemma~\ref{L:h^2=Q}(g).

When $h$ is an involution, all involutions in the coset $Jh$ are $J$-conjugate
by Lemma~\ref{L:h^2=1}(c).  From Lemma~\ref{L:h^2=1}(d), $C_S(h) = \gen{h}
\times B_0$ where $B_0$ is dihedral of order $8 < 2^{k+1}$, whence $|C_S(h)| <
|C_S(b)|$. In the case where $Q = \gen{h^2}$, there are no involutions in $Jh$
by Lemma~\ref{L:h^2=Q}(c).  This shows that in either case the set of fully
$\F$-centralized $\F$-conjugates of $b$ outside $J$ lies in $Ja \cup Jha$. By
Theorem~\ref{T:TT}, there exists a morphism $\phi \in \F$ such that
$b^\phi \in J$ is fully $\F$-centralized, and $C_S(b)^\phi \leq
C_S(b^\phi)$.  Since $C_S(b)$ has nilpotence class $k \geq 3$ and $[S,S,S]
\leq J_0$ from \eqref{E:[S,S]}, we have $(xz)^\phi \in \Omega_1(J_0) =
\gen{x,z}$. It follows that $(xz)^\phi = xz$ as $xz$ is not $\F$-conjugate to
$x$ or to $z$. We may assume by Lemma~\ref{L:a^2=1}(c) that $b^\phi \in
\{x,z\}$. Composing with $c_a$ if necessary, we may assume that $b^\phi =
z$. Thus, we have shown there exist $\phi_a$, $\phi_{ha} \in C_\F(xz)$ such
that $a^{\phi_a} = (tha)^{\phi_{ha}} = z$.

Set $\N = C_\F(xz)$.  We claim that the hyperfocal subgroup $\hyp(\N)$ is of
index $2$ in $S$. We will show this by first demonstrating that the normal
closure $\gen{[a, \phi_{a}]^S, [tha, \phi_{ha}]^S}$ is the commuting product $D
:= D_1 * D_2$ of two quaternion or two semidihedral subgroups of order
$2^{k+1}$ as in Lemma~\ref{L:h^2=1}(i) or Lemma~\ref{L:h^2=Q}(i), respectively.
Then we shall use a transfer argument inside $\N$ to show that in fact
$\foc(\N) = D$ from which it will follow that $\hyp(\N) = D$ as well.

Set $D_0 =\gen{[a, \phi_{a}]^S, [tha, \phi_{ha}]^S} = \gen{(xa)^S, (xtha)^S}$.
Now $xtha\cdot xa = zth \in D_0$.  Since $th$ inverts $J_0$ and all elements of
$J_0th$ are $J$-conjugate (Lemma~\ref{L:h^2=1} and Lemma~\ref{L:h^2=Q}), we
have that $D_0$ contains $J_0$. As $x \in D_0$, $a \in D_0$. Hence $[f,a] =
ff^a \in D_0$.  This shows that $[S,S]=J_0\gen{ff^a}= J_0\gen{[f,a]} \leq
D_0$.  But $D_0$ is a proper normal subgroup of $S$ contained in $D =
[S,S]\gen{th,a}$, and so $D_0 = D$. We conclude that $\foc(\N) \geq D$ is of
index $1$ or $2$ in $S$. 

Write $D = D_1*D_2$ with $D_i$ as in Lemma~\ref{L:h^2=1} or
Lemma~\ref{L:h^2=Q}. As $f$ interchanges $D_1$ and $D_2$, $fz \nin D$.  Suppose
that $\foc(\N) = S$.  Then by Theorem~\ref{T:TT}, there exists a morphism $\eta
\in \N$ such that $(fz)^\eta \in D$ is fully $\N$-centralized and $C_S(fz)^\eta
\leq C_S((fz)^\eta)$.  Since $C_S(fz) = \gen{x, f} \times P$ is of $2$-rank $4$
we have that $(fz)^\eta \in J \cap D = J_0\gen{ff^a}$.  Suppose $(fz)^\eta \in
J_0ff^a$. By Lemma~\ref{L:h^2=1}(h) and Lemma~\ref{L:h^2=Q}(h) then,
$C_S((fz)^\eta) \cong (C_2 \times C_2) \wr C_2$ is of order $2^5$, forcing $|P|
= 8$ and $\eta|_{C_S(fz)}$ to be an isomorphism $C_S(fz) \to C_S((fz)^\eta)$.
But $|Z(C_S(fz))| = 8$ whereas $|Z(C_S((fz)^\eta))| = 4$, a contradiction.
Therefore $(fz)^\eta \in \Omega_1(J_0) = \gen{x,z}$ and $(fz)^\eta = x$ or $z$
because $\eta \in \N$. But $fz$ is $C_\F(z)$-conjugate to $xz$, another
contradiction.  We conclude that $\foc(\N) = D$ is of index $2$ in $S$.  As
$S/\foc(\N)$ is cyclic this shows that $\hyp(\N) = \foc(\N) = D$ is of index
$2$ as well by Lemma~\ref{L:fochyp}(d).

Let $\M = O^2(\N)$, a saturated fusion system on $D$. Set $\M^+ = \M/\gen{xz}$,
and let $\tau\colon \M \to \M^+$ denote the surjective morphism of fusion
systems.  Thus $\M^+$ is a saturated fusion system on $D^+ = D/Z(S)$, a product
$D^+_1 \times D^+_2$ of dihedral groups each of order $2^k$, and with $\M^+ =
O^2(\M^+)$ by Lemma~\ref{L:O^pbasic}. Since $k \geq 3$, $\Aut(D^+)$ is a
$2$-group by Lemma~\ref{L:autDwr2}, and so it follows that $\M^+ =
O^{2'}(\M^+)$ as well, by Proposition~\ref{C:pprimeindexcor}. The hypotheses of
Theorem~\ref{T:OliverSplitting} are now satisfied and therefore $\M^+ \cong
\M_1^+ \times \M_2^+$, by that theorem, for some pair $\M_i^+$ of saturated
fusion systems on $D_i^+$. Note then that the $\M_i^+ = O^2(\M_i^+)$ are
determined as the unique perfect $2$-fusion system on the dihedral group
$D_i^+$ (Lemmas~\ref{L:O^pbasic} and \ref{L:fsmaxcl}), i.e. as
$\F_{D_i^+}(M_i^+)$ with $M_i^+ \cong L_2(q)$.

Let $\M_i$ be the preimage of $\M_i^+$ under $\tau$; $\M_i$ is a saturated
fusion system on $D_i$ by \cite[Lemma~8.10(b)]{AschbacherNormal}. Then $\M_i =
O^2(\M_i)$ since $Z(S) = Z(\M_i) \leq [D_i, D_i]$, and $\M_i/Z(\M_i) = \M_i^+$
for each $i = 1, 2$. As there are no perfect fusion systems on a semidihedral
group with nontrivial center by Lemma~\ref{L:fsmaxcl}(b), each of the $D_i$ is
quaternion, and hence $\M_i$ is the $2$-fusion system of $SL_2(q)$ by (c) of
the same lemma. Furthermore as $f$ interchanges $D_1$ and $D_2$ and $\M_i =
C_\M(D_{3-i})$, $f$ interchanges $\M_1$ and $\M_2$.

In particular, we conclude that $D$ is a commuting product of quaternion groups
of the same order, and $S \cong Q_{2^{k+1}} \wr_* C_2$. Thus $S$ is not
isomorphic to $SD_{2^{k+1}} \wr_* C_2$ as the latter has an involution with
centralizer isomorphic to $C_2 \times SD_{2^{k+1}}$, whereas the former does
not. By Lemma~\ref{L:h^2=1}(i) and Lemma~\ref{L:h^2=Q}(i), $h$ is an
involution. 
\end{proof}

We now extract fusion information from the description of the centralizer
of the central involution in Lemma~\ref{L:centinvcent}.
\begin{lemma}\label{L:h=1}
$h = 1$. 
\end{lemma}
\begin{proof}
Suppose $h \neq 1$. Then the structure of $S$ is that of Lemma~\ref{L:h^2=1}
and $\N = C_\F(xz)$ is given by Lemma~\ref{L:centinvcent}. Let $\N^+ =
\N/\gen{xz}$ as before, and denote passage to the quotient by pluses. Recall $S
= D_i\wr_* C_2$, with $f$ a wreathing element and with each $D_i$ quaternion
and given as in Lemma~\ref{L:h^2=1}(h). We claim
\begin{align}
\text{every element of $J_0ff^a$ is $\N$-conjugate to $x$},
\label{E:J0ff^a}
\end{align}
and once shown, this contradicts Lemma~\ref{L:a^2=1}(a). 

To see \eqref{E:J0ff^a}, note that $Z(S^+) = \gen{x^+}$. From
Lemma~\ref{L:h^2=1}(i), $D_1J_0/J_0$ is a four group generated by the images of
$f^ah$ and $ff^aa$, and $D_2J_0/J_0$ is a four group generated by the images of
$f^ah$ and $a$. Hence, the image in $S^+$ of each element of $J_0ff^a$ is an
involution which is not contained in either of the $D_i^+$ factors of the base
subgroup of $S^+$. Thus, by the structure of $\N^+$, each such element has
image in $S^+$ which is $\N^+$ conjugate to $x^+$.  It follows that each
element of $J_0ff^a$ is $\N$-conjugate into $Z(T) = \gen{x,z}$, and hence
$\N$-conjugate to $x$. This finishes the proof of \eqref{E:J0ff^a} and the
lemma.
\end{proof}

\begin{lemma}\label{L:F=L4q}
$\F$ is the fusion system of $L_4(q_1)$ for some $q_1 \equiv 3 \pmod{4}$ with
$\nu_2(q_1+1) = k-1$. 
\end{lemma}
\begin{proof}
By Lemma~\ref{L:h=1}, $S = J\gen{a}$ is isomorphic to $D_{2^k} \wr C_2$. By
\cite[Proposition~5.5(a)]{OliverSectRank4}, either $k=3$ and $\F$ is the fusion
system of $A_{10}$, or $\F$ is the fusion system of $L_4(q_1)$ for $q_1 \equiv
3 \pmod{4}$ with $\nu_2(q_1+1) = k-1$. In the case of $A_{10}$, $x$ is a
product of two transpositions with centralizer having a unique component $\K$
isomorphic to the fusion system of $A_6 \cong L_2(9)$.  But then $Q =
C_{T}(\K)$ is a four group, contrary to hypothesis. 
\end{proof}

Lemma~\ref{L:F=L4q} completes the identification of $\F$ and the proof of
Theorem~\ref{T:2rank4main}. 

\bibliographystyle{plain}{}
\bibliography{/math/home/fac/jl1474/math/research/mybib}

\def\cprime{$'$}
\begin{thebibliography}{10}

\bibitem{AOV2012}
Kasper K.~S. Andersen, Bob Oliver, and Joana Ventura.
\newblock Reduced, tame and exotic fusion systems.
\newblock {\em Proc. Lond. Math. Soc. (3)}, 105(1):87--152, 2012.

\bibitem{AschbacherFGTSecond}
Michael Aschbacher.
\newblock {\em Finite group theory}, volume~10 of {\em Cambridge Studies in
  Advanced Mathematics}.
\newblock Cambridge University Press, Cambridge, second edition, 2000.

\bibitem{AschbacherNormal}
Michael Aschbacher.
\newblock Normal subsystems of fusion systems.
\newblock {\em Proc. Lond. Math. Soc. (3)}, 97(1):239--271, 2008.

\bibitem{AschbacherGeneration}
Michael Aschbacher.
\newblock Generation of fusion systems of characteristic 2-type.
\newblock {\em Invent. Math.}, 180(2):225--299, 2010.

\bibitem{AschbacherGeneralized}
Michael Aschbacher.
\newblock The generalized {F}itting subsystem of a fusion system.
\newblock {\em Mem. Amer. Math. Soc.}, 209(986):v++110pp., 2011.

\bibitem{AschbacherCopenhagen}
Michael Aschbacher.
\newblock $2$-fusion systems of component type.
\newblock {\em preprint}, page 32pp., 2013.
\newblock Copenhagen lectures; available at
  \url{http://www.math.ku.dk/english/research/conferences/groups2013/Aschbache%
rNotes.pdf}.

\bibitem{AschbacherF2type}
Michael Aschbacher.
\newblock Fusion systems of {$\bold{F}_2$}-type.
\newblock {\em J. Algebra}, 378:217--262, 2013.

\bibitem{AschbacherS3Free}
Michael Aschbacher.
\newblock ${S}_3$-free 2-fusion systems.
\newblock {\em Proceedings of the Edinburgh Mathematical Society (Series 2)},
  56:27--48, 1 2013.

\bibitem{AschbacherFSofCT}
Michael Aschbacher.
\newblock On fusion systems of component type.
\newblock {\em preprint}, page 227pp., 2014.

\bibitem{AschbacherKessarOliver2011}
Michael Aschbacher, Radha Kessar, and Bob Oliver.
\newblock {\em Fusion systems in algebra and topology}, volume 391 of {\em
  London Mathematical Society Lecture Note Series}.
\newblock Cambridge University Press, Cambridge, 2011.

\bibitem{AschbacherSmith2004I}
Michael Aschbacher and Stephen~D. Smith.
\newblock {\em The classification of quasithin groups. {I}}, volume 111 of {\em
  Mathematical Surveys and Monographs}.
\newblock American Mathematical Society, Providence, RI, 2004.
\newblock Structure of strongly quasithin $K$-groups.

\bibitem{BCGLO2007}
C.~Broto, N.~Castellana, J.~Grodal, R.~Levi, and B.~Oliver.
\newblock Extensions of {$p$}-local finite groups.
\newblock {\em Trans. Amer. Math. Soc.}, 359(8):3791--3858 (electronic), 2007.

\bibitem{BrotoLeviOliver2003}
Carles Broto, Ran Levi, and Bob Oliver.
\newblock The homotopy theory of fusion systems.
\newblock {\em J. Amer. Math. Soc.}, 16(4):779--856 (electronic), 2003.

\bibitem{BrotoMollerOliver2012}
Carles Broto, Jesper~M. M{\o}ller, and Bob Oliver.
\newblock Equivalences between fusion systems of finite groups of {L}ie type.
\newblock {\em J. Amer. Math. Soc.}, 25(1):1--20, 2012.

\bibitem{Craven2010}
David~A. Craven.
\newblock Control of fusion and solubility in fusion systems.
\newblock {\em J. Algebra}, 323(9):2429--2448, 2010.

\bibitem{CravenTheory}
David~A. Craven.
\newblock {\em The theory of fusion systems}, volume 131 of {\em Cambridge
  Studies in Advanced Mathematics}.
\newblock Cambridge University Press, Cambridge, 2011.
\newblock An algebraic approach.

\bibitem{Fritz1977}
Franz~J. Fritz.
\newblock On centralizers of involutions with components of {$2$}-rank two. {I}
  and {II}.
\newblock {\em J. Algebra}, 47(2):323--399, 1977.

\bibitem{Gorenstein1980}
Daniel Gorenstein.
\newblock {\em Finite groups}.
\newblock Chelsea Publishing Co., New York, second edition, 1980.

\bibitem{GorensteinHarada1973}
Daniel Gorenstein and Koichiro Harada.
\newblock Finite groups with {S}ylow {$2$}-subgroups of type {${\rm PSp}(4,$}
  {$q),\,q$} odd.
\newblock {\em J. Fac. Sci. Univ. Tokyo Sect. IA Math.}, 20:341--372, 1973.

\bibitem{GLS6}
Daniel Gorenstein, Richard Lyons, and Ronald Solomon.
\newblock {\em The classification of the finite simple groups. {N}umber 6.
  {P}art {IV}}, volume~40 of {\em Mathematical Surveys and Monographs}.
\newblock American Mathematical Society, Providence, RI, 2005.
\newblock The special odd case.

\bibitem{Harris1977}
Morton~E. Harris.
\newblock Finite groups having an involution centralizer with a {$2$}-component
  of dihedral type. {II}.
\newblock {\em Illinois J. Math.}, 21(3):621--647, 1977.

\bibitem{Harris1981}
Morton~E. Harris.
\newblock {${\rm PSL}(2,\,q)$} type {$2$}-components and the unbalanced group
  conjecture.
\newblock {\em J. Algebra}, 68(1):190--235, 1981.

\bibitem{HarrisSolomon1977}
Morton~E. Harris and Ronald Solomon.
\newblock Finite groups having an involution centralizer with a {$2$}-component
  of dihedral type. {I}.
\newblock {\em Illinois J. Math.}, 21(3):575--620, 1977.

\bibitem{Henke2011}
Ellen Henke.
\newblock Minimal fusion systems with a unique maximal parabolic.
\newblock {\em J. Algebra}, 333:318--367, 2011.

\bibitem{Henke2013}
Ellen Henke.
\newblock Products in fusion systems.
\newblock {\em J. Algebra}, 376:300--319, 2013.

\bibitem{LyndTL}
Justin Lynd.
\newblock The {T}hompson-{L}yons transfer lemma for fusion systems.
\newblock {\em Bull. Lond. Math. Soc.}, 2014.
\newblock 7pp. (to appear); arXiv:1303.5996 [math.GR].

\bibitem{Mason1973}
David~R. Mason.
\newblock Finite simple groups with {S}ylow {$2$}-subgroup dihedral wreath
  {$Z_{2}$}.
\newblock {\em J. Algebra}, 26:10--68, 1973.

\bibitem{Oliver2006}
Bob Oliver.
\newblock Equivalences of classifying spaces completed at the prime two.
\newblock {\em Mem. Amer. Math. Soc.}, 180(848):vi+102pp., 2006.

\bibitem{OliverSplitting}
Bob Oliver.
\newblock Splitting fusion systems over 2-groups.
\newblock {\em Proceedings of the Edinburgh Mathematical Society (Series 2)},
  56:263--301, 1 2013.

\bibitem{OliverSectRank4}
Bob Oliver.
\newblock Reduced fusion systems over $2$-groups of sectional rank at most
  four.
\newblock {\em Mem. Amer. Math. Soc.}, pages iii+96, 2014.
\newblock (to appear).

\bibitem{Puig2000}
Lluis Puig.
\newblock The hyperfocal subalgebra of a block.
\newblock {\em Invent. Math.}, 141(2):365--397, 2000.

\bibitem{Puig2006}
Lluis Puig.
\newblock Frobenius categories.
\newblock {\em J. Algebra}, 303(1):309--357, 2006.

\bibitem{Suzuki1982}
Michio Suzuki.
\newblock {\em Group theory. {I}}, volume 247 of {\em Grundlehren der
  Mathematischen Wissenschaften [Fundamental Principles of Mathematical
  Sciences]}.
\newblock Springer-Verlag, Berlin, 1982.
\newblock Translated from the Japanese by the author.

\bibitem{Suzuki1986}
Michio Suzuki.
\newblock {\em Group theory. {II}}, volume 248 of {\em Grundlehren der
  Mathematischen Wissenschaften [Fundamental Principles of Mathematical
  Sciences]}.
\newblock Springer-Verlag, New York, 1986.
\newblock Translated from the Japanese.

\bibitem{Welz2012}
Matthew Welz.
\newblock Fusion systems with standard components of small rank.
\newblock 2012.
\newblock PhD thesis, University of Vermont.

\end{thebibliography}

\end{document}